\documentclass[11pt,a4paper]{scrartcl}
\usepackage[a4paper,margin=2.6cm]{geometry}
\usepackage{amsmath,amssymb,amsthm,mathtools}
\usepackage{fontspec}                 
\usepackage{microtype}
\usepackage[hidelinks]{hyperref}
\usepackage[capitalise,noabbrev]{cleveref}
\usepackage{booktabs}
\usepackage{longtable,array,calc}
\usepackage{graphicx}
\usepackage[most]{tcolorbox}
\tcbuselibrary{theorems}
\usepackage{varwidth}                 
\usepackage{tikz}
\usetikzlibrary{arrows.meta,positioning,shapes.geometric}

\makeatletter
\renewenvironment{thebibliography}[1]
  {\list{\@biblabel{\@arabic\c@enumiv}}%
        {\settowidth\labelwidth{\@biblabel{#1}}%
         \leftmargin\labelwidth
         \advance\leftmargin\labelsep
         \@openbib@code
         \usecounter{enumiv}%
         \let\p@enumiv\@empty
         \renewcommand\theenumiv{\@arabic\c@enumiv}}%
   \sloppy
   \clubpenalty4000
   \@clubpenalty \clubpenalty
   \widowpenalty4000%
   \sfcode`\.\@m}
  {\def\@noitemerr
    {\@latex@warning{Empty `thebibliography' environment}}%
   \endlist}
\providecommand{\@openbib@code}{}
\makeatother

\tcbset{varwidth boxed title*=-1cm}

\definecolor{clrThm}{HTML}{2F6FAF}
\definecolor{clrDef}{HTML}{1F7A5A}
\definecolor{clrRem}{HTML}{8A6D1F}
\definecolor{clrExa}{HTML}{6B4C9A}

\newtcbtheorem{theorem}{Theorem}{%
  enhanced, breakable, colback=clrThm!4, colframe=clrThm, coltitle=white,
  fonttitle=\bfseries, arc=2pt, boxrule=0.6pt, left=6pt, right=6pt,
  attach boxed title to top left={xshift=6pt,yshift=-2mm},
  boxed title style={colback=clrThm, arc=1pt}}{}

\newtcbtheorem[use counter from=theorem]{proposition}{Proposition}{%
  enhanced, breakable, colback=clrThm!4, colframe=clrThm, coltitle=white,
  fonttitle=\bfseries, arc=2pt, boxrule=0.6pt, left=6pt, right=6pt,
  attach boxed title to top left={xshift=6pt,yshift=-2mm},
  boxed title style={colback=clrThm, arc=1pt}}{}
\newtcbtheorem[use counter from=theorem]{lemma}{Lemma}{%
  enhanced, breakable, colback=clrThm!4, colframe=clrThm, coltitle=white,
  fonttitle=\bfseries, arc=2pt, boxrule=0.6pt, left=6pt, right=6pt,
  attach boxed title to top left={xshift=6pt,yshift=-2mm},
  boxed title style={colback=clrThm, arc=1pt}}{}
\newtcbtheorem[use counter from=theorem]{definition}{Definition}{%
  enhanced, breakable, colback=clrDef!4, colframe=clrDef, coltitle=white,
  fonttitle=\bfseries, arc=2pt, boxrule=0.6pt, left=6pt, right=6pt,
  attach boxed title to top left={xshift=6pt,yshift=-2mm},
  boxed title style={colback=clrDef, arc=1pt}}{}
\newtcbtheorem[use counter from=theorem]{remark}{Remark}{%
  enhanced, breakable, colback=clrRem!4, colframe=clrRem, coltitle=white,
  fonttitle=\bfseries, arc=2pt, boxrule=0.6pt, left=6pt, right=6pt,
  attach boxed title to top left={xshift=6pt,yshift=-2mm},
  boxed title style={colback=clrRem, arc=1pt}}{}
\newtcbtheorem[use counter from=theorem]{example}{Example}{%
  enhanced, breakable, colback=clrExa!4, colframe=clrExa, coltitle=white,
  fonttitle=\bfseries, arc=2pt, boxrule=0.6pt, left=6pt, right=6pt,
  attach boxed title to top left={xshift=6pt,yshift=-2mm},
  boxed title style={colback=clrExa, arc=1pt}}{}

\definecolor{clrPrb}{HTML}{A0522D}
\newtcbtheorem[use counter from=theorem]{problem}{Problem}{%
  enhanced, breakable, colback=clrPrb!4, colframe=clrPrb, coltitle=white,
  fonttitle=\bfseries, arc=2pt, boxrule=0.6pt, left=6pt, right=6pt,
  attach boxed title to top left={xshift=6pt,yshift=-2mm},
  boxed title style={colback=clrPrb, arc=1pt}}{}

\definecolor{clrQue}{HTML}{B03060}
\newtcbtheorem{question}{Question}{%
  enhanced, breakable, colback=clrQue!4, colframe=clrQue, coltitle=white,
  fonttitle=\bfseries, arc=2pt, boxrule=0.6pt, left=6pt, right=6pt,
  attach boxed title to top left={xshift=6pt,yshift=-2mm},
  boxed title style={colback=clrQue, arc=1pt}}{}

\newcommand{\thmref}[2]{\hyperref[:#2]{#1~\ref*{:#2}}}

\makeatletter
\newsavebox{\jts@fitbox}
\newcommand{\fitdisplay}[1]{%
  \sbox{\jts@fitbox}{\ensuremath{\displaystyle #1}}%
  \ifdim\wd\jts@fitbox>\linewidth
    \resizebox{\linewidth}{!}{\usebox{\jts@fitbox}}%
  \else
    \usebox{\jts@fitbox}%
  \fi}
\makeatother

\let\jtsOldUnderscore\_
\renewcommand{\_}{\jtsOldUnderscore\allowbreak}

\begin{document}

%
%

\definecolor{orcidgreen}{HTML}{A6CE39}
\newcommand{\orcidid}[1]{%
  \href{https://orcid.org/#1}{%
    \textcolor{orcidgreen}{\fontsize{8}{9}\selectfont iD}%
    \,\fontsize{8}{9}\selectfont #1}}

\begingroup
\centering
\leftskip=0.05\textwidth plus 1fil \rightskip=0.05\textwidth plus 1fil
    \fontsize{15}{18}\selectfont\bfseries
    The coordinate ring of the $k$-fold iterated commutator locus for
    $2 \times 2$ matrices%
    \par
\endgroup

\begingroup
\centering
\vspace{0.8em}
{\fontsize{11}{13}\selectfont Jan Snellman\hspace{0.25em}%
  \orcidid{0009-0002-6676-5068}\par}

\vspace{0.35em}

{\fontsize{9}{11}\selectfont Matematiska Institutionen, Linköpings
  Universitet, 581 83 Linköping, Sweden\par}

\vspace{0.15em}

{\fontsize{9}{11}\selectfont
  \href{mailto:jan.snellman@liu.se}{jan.snellman@liu.se}\par}
\endgroup

\vspace{1.0em}

\begin{center}
\begin{minipage}{\dimexpr\textwidth-3.2cm\relax}
  \fontsize{9.2}{11.5}\selectfont
  \begin{center}\textbf{Abstract}\end{center}
  \vspace{0.25em}
  For $2 \times 2$ matrices $A_1,\dots,A_k$, write $[A_1,\dots,A_k]$ for the
  left-normed iterated commutator $[\dots[[A_1,A_2],A_3],\dots,A_k]$, and $I_k$
  for the ideal, in the $3k$-variable reduced-coordinate polynomial ring $R_k$,
  cutting out its vanishing locus. We prove, for every $k \geq 2$ over any field
  of characteristic $\neq 2$, and as four independently-established results
  rather than one bundled claim: $I_k$ has codimension 2; $I_k$ has exactly 3
  minimal generators; $R_k / I_k$ is Cohen-Macaulay; and $I_k$ is radical. The
  last of these, together with an explicit component count resting on a
  non-containment argument, assembles into the \textbf{Primary Decomposition
  Theorem}: $I_k = P_2 \cap \cdots \cap P_k$ is an irredundant primary
  decomposition into exactly $k-1$ primes, following an explicit recursive
  block-involvement pattern. The proof identifies $I_k$ as the ideal of
  $2 \times 2$ minors of an explicit $2 \times 3$ matrix (a determinantal ideal,
  not merely determinantal-\textit{looking}), and invokes classical
  determinantal-ideal theory (Bruns-Vetter) and an explicit rank-2 Jacobian
  witness on every component (Serre's criterion) for the algebraic and
  radicality halves respectively.

  \vspace{0.9em}
  {\fontsize{8.6}{10.5}\selectfont
   \par}

  \vspace{0.3em}
  {\fontsize{8.6}{10.5}\selectfont
   \par}
\end{minipage}
\end{center}

\vspace{0.8em}

\section{Introduction}

\(2 \times 2\) matrices \(A,B\) have commutator \(\lbrack A,B\rbrack = AB - BA\), and the \textbf{iterated} left-normed commutator of \(k\) matrices is \(\left\lbrack A_{1},\ldots,A_{k} \right\rbrack \coloneqq  \left\lbrack \ldots\left\lbrack \left\lbrack A_{1},A_{2} \right\rbrack,A_{3} \right\rbrack,\ldots,A_{k} \right\rbrack\), matching the standard lower-central-series convention for Lie algebras. This object arose out of

\begin{itemize}
\item
  a study of the corresponding formal word in the theory of fiber-counting of word-maps \(\pi:F_{n} \rightarrow G\), where \(F_{n}\) is the free group on \(n\) letters; in particular, after a close reading of Tambour's result \cite{Tambour2000}
\item
  a question in algebraic complexity theory -- how many scalar multiplications are needed to \textbf{compute} \(\left\lbrack A_{1},\ldots,A_{k} \right\rbrack\)? -- which remains genuinely open for \(k \geq 3\) and which the author is presently investigating.
\end{itemize}

The one fact from that story needed here is purely algebraic, not complexity-theoretic: an explicit bilinear formula for the ordinary \(2 \times 2\) commutator, chained across \(k\) matrices, gives an explicit polynomial parametrization of \(\left\lbrack A_{1},\ldots,A_{k} \right\rbrack\)'s reduced coordinates (\thmref{Definition}{def-wk} below), and it is the \textbf{vanishing locus} of this parametrization, as an algebraic variety and its coordinate ring, that is the actual subject of this article. One naming note: one could alternatively refer to the coordinate ring of the variety here studied as the ``ring-of-k-commuting-2x2-matrices'', using ``commuting'' in this generalized iterated-commutator-vanishing sense. This means that \(\left\lbrack A_{1},\ldots,A_{k} \right\rbrack = 0\), and not literally ``every pair of \(A_{1},\ldots,A_{k}\) commutes''. The two concepts agree at \(k = 2\), where this article's locus \textbf{is} the classical commuting variety; for \(k \geq 3\) the vanishing of the iterated bracket is the \textbf{weaker} of the two conditions, so the locus studied here properly contains the pairwise-commuting one.

\textbf{The classical commuting variety.} For \(A,B\) arbitrary (not necessarily \(2 \times 2\)) \(n \times n\) matrices, the variety \(C_{n} \coloneqq  \left\{ (A,B) \in M_{n} \times M_{n}:AB = BA \right\}\) -- the ordinary, unreduced, un-iterated case this article's own \(k = 2\) specializes to -- has been studied since the 1950s. Its irreducibility was proven independently by Motzkin and Taussky \cite{MotzkinTaussky1955} and by Gerstenhaber \cite{Gerstenhaber1961}, the latter also establishing the now-standard bound \(\dim_{\mathbb{F}}\mathbb{F}\lbrack A,B\rbrack \leq n\) (here \(\mathbb{F}\) is a field) for the subalgebra generated by two commuting matrices. Whether the coordinate ring \(\mathbb{F}\lbrack A,B\rbrack/I\), \(I\) the ideal generated by the entries of \(AB - BA\), is \textbf{reduced} and \textbf{Cohen-Macaulay} for every \(n\) is a much harder question, attributed to Michael Artin and Melvin Hochster -- an unpublished, informally-circulated conjecture with no single dateable original source \cite{ArtinHochsterConjecture}, stated precisely as ``Conjecture 1.1'' in Majidi-Zolbanin and Snapp \cite{MajidiZolbaninSnapp}. Chau \cite{Chau2024} records a more specifically dated 1984 conjecture, attributed to Hochster alone and asking only for \textbf{reducedness}, about these same commuting loci -- settled at \(n = 3\) by Thompson shortly afterwards, and open for \(n \geq 4\). Freyja Hreinsdóttir's manuscript \cite{Hreinsdottir2005}, building on her 1997 Stockholm doctoral thesis under Jan-Erik Roos, attacks the Artin-Hochster conjecture case by case at \(n = 3,4\): syzygies, Betti numbers, the canonical module, and a Koszul-dual structure theorem, verified computationally through \(n = 7\) but proven in general only in these small cases. Knutson \cite{Knutson2003} studies the minimal primes of closely related schemes from a different angle. The state of play for the \textbf{classical} commuting variety at the time of writing is therefore: Cohen-Macaulayness is

\begin{itemize}
\item
  settled by hand at \(n \leq 2\);
\item
  settled by machine computation at \(n = 3\) and \(n = 4\), as both Majidi-Zolbanin and Snapp \cite{MajidiZolbaninSnapp} and Hreinsdóttir \cite{Hreinsdottir2005} record; and
\item
  open for every \(n \geq 5\).
\end{itemize}

Reducedness is settled for \(n \leq 3\) and open for \(n \geq 4\) \cite{Chau2024}.

At \(n = 2\) specifically -- this article's \(k = 2\) base case -- the question is comparatively easy, and is in effect answered again, as a special case, by \thmref{Theorem}{thm-cm} below: \(I_{2}\) (\thmref{Definition}{def-wk}) is literally the ideal of the \(n = 2\) commuting variety, expressed in reduced (trace-removed) coordinates rather than the full matrix-entry coordinates the classical literature typically uses. The genuinely new object in this article is the \textbf{iterated} generalization to \(k \geq 3\): not the classical commuting variety at all (see the naming note above), but a structurally analogous determinantal ideal, built by chaining the same bilinear cross-product formula that governs the \(n = 2\) case. It is exactly the same \textbf{flavor} of question the Artin-Hochster conjecture asks. Is this natural, geometrically meaningful ideal Cohen-Macaulay? Is it reduced? Here those two questions are asked of a different, recursively-defined family of ideals, and this article answers them completely and unconditionally, for every \(k \geq 2\) (\thmref{Theorem}{thm-primarydecomp}), rather than case by case or conjecturally. In that sense the results below can be read as a fully worked structural analogue of the still-open classical question, in a setting simple enough (\(n = 2\), but \(k\) unbounded) to admit a complete answer.

\section{Notation}

We present the notation used across more than one later result, in order of first appearance. Symbols local to a single proof are omitted. Throughout, \(\mathbb{F}\) denotes the ambient field; it is fixed in \cref{sec-setup} to be of characteristic \(\neq 2\), and nothing in this section depends on that.

\textbf{The following typographic convention is used throughout.} An \textbf{overline} denotes a \textbf{vector} -- an element of \(\mathbb{F}^{3}\) (or, for \(R_{k}\) itself, of the free module \(R_{k}^{3}\)) -- while the \textbf{same letter, plain, with a two-index comma subscript}, denotes one of that vector's three individual scalar \textbf{components}. So \({\overline{v}}_{i} \in \mathbb{F}^{3}\) is block \(i\)'s whole reduced-coordinate vector, while its three components are the plain scalars \(v_{i,1},v_{i,2},v_{i,3} \in \mathbb{F}\) (equivalently \(v_{i,j}\) for \(j = 1,2,3\)) -- never overlined, since a single scalar entry is not itself a vector. The same convention applies to \({\overline{w}}_{j}\) versus its components \(w_{j,1},w_{j,2},w_{j,3}\).

\begin{longtable}[]{@{}>{\centering\arraybackslash}p{\dimexpr(\linewidth-6\tabcolsep)/3\relax}>{\raggedright\arraybackslash}p{\dimexpr(\linewidth-6\tabcolsep)/3\relax}>{\raggedright\arraybackslash}p{\dimexpr(\linewidth-6\tabcolsep)/3\relax}@{}}
\toprule\noalign{}
\endhead
\bottomrule\noalign{}
\endlastfoot
\textbf{Symbol} & \textbf{Meaning} & \textbf{Defined in} \\
\(\overline{v}\) (also \({\overline{v}}_{i}\); components \(v_{i,1},v_{i,2},v_{i,3}\), or \(v_{i,j}\)) & reduced-coordinate vector in \(\mathbb{F}^{3}\), and its blocks & \S{}3, Setup \\
\(S\) & the sign twist \(S = \text{ diag}( - 1, - 1,1)\), used both as a map and as a matrix & \S{}3, Setup (\thmref{Theorem}{thm-basic-identity}) \\
\(\mathbb{F}\) & ambient field, characteristic \(\neq 2\); \(\overline{\mathbb{F}}\) a fixed algebraic closure & \S{}3, opening paragraph; \(\overline{\mathbb{F}}\) at \thmref{Remark}{rem-base-field} \\
\(R_{k}\) & the polynomial ring \(\mathbb{F}\left\lbrack {\overline{v}}_{1},\ldots,{\overline{v}}_{k} \right\rbrack\) & \thmref{Definition}{def-wk} \\
\(I_{k}\) & the defining ideal, \(I_{k} \subset R_{k}\) & \thmref{Definition}{def-wk} \\
\({\overline{w}}_{j}\) & (a) the recursive vector of \textbf{polynomials}, an element of \(R_{j}^{3}\) (its \textbf{evaluations} lie in \(\mathbb{F}^{3}\)); (b) as a polynomial map \(\mathbb{A}^{3j} \rightarrow \mathbb{A}^{3}\) & (a) \thmref{Definition}{def-wk}; (b) \thmref{Lemma}{lem-nondeg} \\
\(U_{j}\) & the Zariski-open locus where \({\overline{w}}_{j} \neq 0\) & \thmref{Lemma}{lem-nondeg} \\
\(\pi\) & the projection \(\mathbb{A}^{3k} \rightarrow \mathbb{A}^{3(k - 1)}\) forgetting \({\overline{v}}_{k}\) & \thmref{Definition}{def-y1y2} \\
\(\Phi\) & the map \(\mathbb{A}^{3(k - 1)} \times \mathbb{A}^{1} \rightarrow \mathbb{A}^{3k}\) & \thmref{Definition}{def-y1y2} \\
\(Y_{1}\), \(Y_{2}\) & the two strata of \(V\left( I_{k} \right)\) & \thmref{Definition}{def-y1y2} \\
\(P_{j}\) (\(j = 2,\ldots,k\)) & the primary components of \(I_{k}\) & \thmref{Definition}{def-primecomponents} \\
\(J_{k}\) & the Jacobian of \({\overline{w}}_{k}\) & \thmref{Lemma}{lem-jacrecursion} \\
\(\lbrack u\rbrack_{\times}\) & skew matrix with \(\lbrack u\rbrack_{\times}v = u \times v\) & \thmref{Lemma}{lem-jacrecursion} \\
\end{longtable}

\section{Setup: the signed encoding and the defining ideal}\label{sec-setup}

Throughout, \(\mathbb{F}\) is a field of characteristic \(\neq 2\), and all matrices, vectors and polynomial rings are over \(\mathbb{F}\) unless said otherwise. (Where exactly that characteristic hypothesis is spent is a question in its own right, taken up in \thmref{Problem}{prob-char2} at the end of this section; the displayed examples below are computed over \(\mathbb{Q}\), but nothing in the definitions is special to it.)

\textbf{Step 1: discard the scalar part.} Write \(A = \begin{pmatrix}
a_{11} & a_{12} \\
a_{21} & a_{22}
\end{pmatrix}\) as the sum of its scalar part \(\left( \text{tr}(A)/2 \right)I\) and its traceless remainder \begin{equation}\fitdisplay{A_{0} \coloneqq  A - \left( \text{tr}(A)/2 \right)I = \begin{pmatrix}
d_{a}/2 & a_{12} \\
a_{21} & - d_{a}/2
\end{pmatrix},\quad d_{a} \coloneqq  a_{11} - a_{22}.}\end{equation} Scalar matrices are central, so they never contribute to a commutator: \begin{equation}\fitdisplay{\lbrack cI,B\rbrack = cIB - BcI = cB - cB = 0}\end{equation} for every \(B\) and every scalar \(c\), hence \(\lbrack A,B\rbrack = \left\lbrack A_{0},B_{0} \right\rbrack\) identically -- no information relevant to \textbf{any} bracket is lost by discarding the scalar part, so from here on only the traceless remainder matters.

\textbf{Step 2: encode the traceless remainder.} A traceless matrix \(\begin{pmatrix}
x & y \\
z & - x
\end{pmatrix}\) has three free entries. Rather than reading them off directly (\(x,y,z\) themselves -- shown in \thmref{Remark}{rem-why-D} below to be a strictly \textbf{worse} choice, forcing a swap \textbf{and} two separate factors of 2), take \begin{equation}\fitdisplay{\text{ enc}(A) \coloneqq  \left( a_{11} - a_{22},a_{12} + a_{21},a_{12} - a_{21} \right) \in \mathbb{F}^{3},}\end{equation} the traceless \((1,1)\)-entry doubled. On a traceless matrix written as above this reads \begin{equation}\fitdisplay{\text{ enc }\begin{pmatrix}
x & y \\
z & - x
\end{pmatrix} = (2x,y + z,y - z),}\end{equation} and both forms are worth having in view: the first is what one applies to a general \(A\), the second is what one reads when a point of \(\mathbb{A}^{3}\) has to be produced or checked by hand. (On the traceless part \(a_{22} = - a_{11}\), so \(d_{a} = a_{11} - a_{22} = 2a_{11}\): the quantity \(d_{a}\) of Step 1 is literally twice that entry.) The remaining two coordinates are \(A\)'s two off-diagonal entries replaced by their sum and difference. As with the naive encoding \(\text{red}\) of \thmref{Remark}{rem-why-D}, no division is needed to \textbf{compute} \(\text{enc}\) -- only its inverse (recovering \(a_{12},a_{21}\) individually as \(\frac{v_{2} + v_{3}}{2}\) and \(\frac{v_{2} - v_{3}}{2}\)) needs \(2\) invertible, which is the standing hypothesis the article already carries, not a new one.

\textbf{Step 3: form the commutator.} Given traceless \(A_{0} = \begin{pmatrix}
x & y \\
z & - x
\end{pmatrix}\) and \(B_{0} = \begin{pmatrix}
u & v \\
w & - u
\end{pmatrix}\), the ordinary matrix commutator \(\left\lbrack A_{0},B_{0} \right\rbrack = A_{0}B_{0} - B_{0}A_{0}\) is again traceless automatically, since \(\text{tr}(PQ) = \text{tr}(QP)\) for any \(P,Q\) and therefore \(\text{tr}\left( \lbrack P,Q\rbrack \right) = 0\) always, so it is again encodable by \(\text{enc}\).

\textbf{Step 4: the encoded bracket is the cross product, up to one fixed sign twist.}

For the following theorem, the article's standing hypothesis \(\text{char }\mathbb{F} \neq 2\) is deliberately not used: the identity holds over every commutative ring.

\begin{theorem}{the basic identity}{thm-basic-identity} For traceless \(2 \times 2\) matrices \(A_{0},B_{0}\) over \(\mathbb{F}\), with no hypothesis on the characteristic, \begin{equation}\fitdisplay{\text{ enc}\left( \left\lbrack A_{0},B_{0} \right\rbrack \right) = {S\left( \text{enc}\left( A_{0} \right) \times \text{ enc}\left( B_{0} \right) \right)},\quad S \coloneqq  \begin{pmatrix}
 - 1 & 0 & 0 \\
0 & - 1 & 0 \\
0 & 0 & 1
\end{pmatrix}.}\end{equation}

\end{theorem}

\begin{proof}[{Proof of \thmref{Theorem}{thm-basic-identity}}] The proof is a direct computation, carried out symbolically over generic entries \(a_{11},\ldots,b_{22}\) rather than checked at sampled points, in anc/code/verify\_lorentzian\_encoding\_identity.sage. Both sides, expanded in the 8 free entries, agree term by term; they are displayed in full in \cref{app-identity}.

Both sides are integer polynomials in those entries, so the identity descends to every commutative ring. It was also checked directly over \(\mathbb{Z}\) and, separately, over \(\mathbb{F}_{2}\) --- where it still holds --- in anc/code/referee\_round4\_setup\_determinantal\_check.sage, precisely to confirm that this step is not one of the places where \(\text{char }\mathbb{F} \neq 2\) is spent (\thmref{Question}{qst-char2}).

\end{proof}

\begin{remark}{why this particular twist, and not some other one}{rem-why-D} There are three concrete respects in which \(S\) improves on the naive reshuffle \(\tau(a,b,c) = (b,a,2c)\) that the encoding \(\text{red}(A) \coloneqq  \left( a_{12},a_{21},d_{a} \right)\) forces, each checked, not merely asserted:

\begin{enumerate}
\item
  \textbf{\(S\) does not mix coordinates.} \(\tau\) swaps two of its three inputs; \(S\) acts diagonally, one coordinate at a time. Consequence: \(\text{red}\) would need a page-length remark (``why \({\overline{w}}_{j}\) carries no outer \(\tau\)'') to justify that rescaling+swapping a generating set by an invertible linear map doesn't change the ideal generated. Here the analogous fact is immediate: \textbf{individually} negating two of three generators obviously leaves the ideal unchanged, generator by generator -- no swap to account for. (Still verified computationally, not merely asserted: \(\text{ideal}\left( {\overline{w}}_{2} \right) = \text{ ideal}\left( S{\overline{w}}_{2} \right)\), confirmed by direct Gröbner-basis equality in anc/code/verify\_encoding\_gl3\_equivariance.sage the three generators are written out in \cref{app-remark2}.)
\item
  \textbf{\(S\) is an involution; \(\tau\) is not.} \(S^{2} = I\) exactly (immediate from the diagonal entries being \(\pm 1\)), while \(\tau^{2} = \text{ diag}(1,1,4) \neq I\).
\item
  \textbf{\(S\) is not an arbitrary matrix that happens to work -- it is (up to a constant) the matrix of \(\mathfrak{sl}_{2}\)'s own Killing form, in these coordinates.} For traceless \(A_{0} = \begin{pmatrix}
  x & y \\
  z & - x
  \end{pmatrix}\), \(\det(A_{0}) = - x^{2} - yz\), an \textbf{indefinite} quadratic form (never definite: \(- yz = {- \left( (y + z)/2 \right)}^{2} + \left( (y - z)/2 \right)^{2}\) is a genuine difference of squares whenever \(2 \neq 0\)). Checked directly: \(- {\text{enc}(A)}_{1}^{2} - {\text{ enc}(A)}_{2}^{2} + {\text{ enc}(A)}_{3}^{2} = 4\det(A_{0})\), exactly (Sage, symbolic; \cref{app-remark2}). So \(S\)'s sign pattern \textbf{is} \(\det\)'s own signature, read off in these coordinates -- not a bookkeeping accident. (Lie-theoretically: \(\mathfrak{sl}_{2}\left( \mathbb{R} \right) \cong \mathfrak{so}(2,1)\), the Lorentz algebra, not the compact \(\mathfrak{so}(3)\) one gets only after adjoining \(i\) -- classical, not cited to a specific source here, the same way this article leaves certain folklore facts uncited rather than inventing a citation.) Over \(\mathbb{R}\), or over any formally real field such as \(\mathbb{Q}\), that sign cannot be removed by \textbf{any} invertible linear re-encoding: \(- x^{2} - y^{2} + z^{2}\) is isotropic, vanishing at \((1,0,1)\), whereas \(x^{2} + y^{2} + z^{2}\) is anisotropic, since over a formally real field a sum of squares vanishes only when every term does -- and isotropy is manifestly preserved by invertible linear substitution. (This is the content of Sylvester's law of inertia in the case at hand; the law itself is a statement about \(\mathbb{R}\).) In that precise sense \(S\)'s single-sign twist is already the simplest one achievable, not merely simpler than \(\tau\) by luck.

  Over a general field the claim genuinely fails, and it fails already in the smallest interesting case, so the restriction above is not idle caution: over \(\mathbb{F}_{3}\), where \(- 1\) is \textbf{not} a square, \(\text{diag}( - 1, - 1,1)\) is nevertheless congruent to the identity form, via \(T = \begin{pmatrix}
  0 & 0 & 1 \\
  1 & 1 & 0 \\
  1 & 2 & 0
  \end{pmatrix}\) (of determinant \(1\)) with \(T^{\top}T = \text{ diag}(2,2,1) = \text{ diag}( - 1, - 1,1)\). So ``\(- 1\) is not a square'' is \textbf{not} the right criterion; formal reality is.
\end{enumerate}

The ``obvious'' un-permuted reading of the matrix entries, \(\text{raw}(A) \coloneqq  (x,y,z)\), is worse than either encoding: the twist it forces is \(\text{swap}(2,3) \cdot \text{ diag}(1,2,2)\) -- a coordinate swap \textbf{and} two separate factors of \(2\). A different escape, adjoining \(i\) to get a literal untwisted cross product (\(\mathfrak{su}(2) \cong \mathfrak{so}(3)\) in place of \(\mathfrak{so}(2,1)\)), is a genuine alternative but a strictly narrower one, since it constrains the field.

\end{remark}

\begin{definition}{reduced iterated commutator, signed encoding}{def-wk} \begin{equation}\fitdisplay{{\overline{w}}_{2} \coloneqq  {\overline{v}}_{1} \times {\overline{v}}_{2},\quad{\overline{w}}_{j} \coloneqq  {S\left( {\overline{w}}_{j - 1} \right)} \times {\overline{v}}_{j}\quad(j = 3,\ldots,k).}\end{equation} Let \(R_{k} \coloneqq  \mathbb{F}\left\lbrack {\overline{v}}_{1},\ldots,{\overline{v}}_{k} \right\rbrack\) (\({\overline{v}}_{i} \in \mathbb{F}^{3}\), \(3k\) variables) and \(I_{k} \subset R_{k}\) the ideal generated by the three components of \({\overline{w}}_{k}\) -- the ideal cutting out \(\left\lbrack A_{1},\ldots,A_{k} \right\rbrack = 0\) in reduced coordinates. Two deliberate understatements are worth flagging. First, \(I_{k}\) is at this point only the \textbf{generated} ideal; that it is also the full vanishing ideal of that locus is \thmref{Theorem}{thm-radical}, in \S{}7 below, and is not to be assumed before then. Second, the encoded bracket is strictly \(S\left( {\overline{w}}_{k} \right)\), not \({\overline{w}}_{k}\) -- but \(S\) merely negates two of three coordinates, so the two triples generate the same ideal, and it is cleaner to carry the untwisted one.

\end{definition}

To spell out exactly what this means: ``\(R_{k} \coloneqq  \mathbb{F}\left\lbrack {\overline{v}}_{1},\ldots,{\overline{v}}_{k} \right\rbrack\)'' is shorthand for \textbf{the polynomial ring in the \(3k\) scalar variables \(v_{i,j}\) (\(i = 1,\ldots,k\), \(j = 1,2,3\))}, grouped into \(k\) blocks of 3 -- \textbf{not} a ring with \(k\) vector-valued generators (a polynomial ring's generators are always scalars). At \(k = 2\), written out in full: \begin{equation}\fitdisplay{R_{2} = \mathbb{F}\left\lbrack v_{1,1},v_{1,2},v_{1,3},v_{2,1},v_{2,2},v_{2,3} \right\rbrack,}\end{equation} six ordinary scalar indeterminates, with \({\overline{v}}_{1} = \left( v_{1,1},v_{1,2},v_{1,3} \right)\) and \({\overline{v}}_{2} = \left( v_{2,1},v_{2,2},v_{2,3} \right)\) merely a notational grouping of them into two blocks of three -- convenient because \({\overline{w}}_{2} = {\overline{v}}_{1} \times {\overline{v}}_{2}\) is most naturally written as a cross product of two block-vectors, not as some formula in six unrelated scalars.

Each \({\overline{w}}_{j}\) is a triple of polynomials, i.e. \({\overline{w}}_{j} \in R_{j}^{3}\) (the free rank-3 module over \(R_{j}\)) -- equivalently, exactly as recorded in the Notation table's entry (b) for \({\overline{w}}_{j}\), the polynomial map \({\overline{w}}_{j}:\mathbb{A}^{3j} \rightarrow \mathbb{A}^{3}\) \textbf{given by substitution}: evaluating each of \({\overline{w}}_{j}\)'s three polynomial components at a point \(\left( {\overline{v}}_{1},\ldots,{\overline{v}}_{j} \right) \in \mathbb{A}^{3j}\). Concretely at \(k = 2\), writing out all six variables, \({\overline{w}}_{2}\) is the substitution map \begin{equation}\fitdisplay{\begin{pmatrix}
v_{1,1} \\
v_{1,2} \\
v_{1,3} \\
v_{2,1} \\
v_{2,2} \\
v_{2,3}
\end{pmatrix} \in \mathbb{A}^{6}\quad \mapsto \quad\begin{pmatrix}
v_{1,2}v_{2,3} - v_{1,3}v_{2,2} \\
v_{1,3}v_{2,1} - v_{1,1}v_{2,3} \\
v_{1,1}v_{2,2} - v_{1,2}v_{2,1}
\end{pmatrix} \in \mathbb{A}^{3},}\end{equation} and \(I_{k}\) (\thmref{Definition}{def-wk}) is, equivalently, the ideal generated by the three coordinate functions of \({\overline{w}}_{k}:\mathbb{A}^{3k} \rightarrow \mathbb{A}^{3}\) pulled back along this substitution.

\begin{example}{\({\overline{w}}_{k}\) and \(I_{k}\), explicitly, for \(k = 2,3,4\)}{ex-wk-signed} Everything below was computed directly (anc/code/compute\_signed\_encoding\_examples.sage) rather than derived by hand. At \(k = 2\), the three components of \({\overline{w}}_{2} = {\overline{v}}_{1} \times {\overline{v}}_{2}\): \begin{equation}\fitdisplay{v_{1,2}v_{2,3} - v_{1,3}v_{2,2},\quad v_{1,3}v_{2,1} - v_{1,1}v_{2,3},\quad v_{1,1}v_{2,2} - v_{1,2}v_{2,1}.}\end{equation} (This is what \(\text{red}\) would give at this stage too -- the two encodings agree exactly at the very first stage, since neither \(\tau\) nor \(S\) is applied yet.) At \(k = 3\), the three components of \({\overline{w}}_{3} = {S\left( {\overline{w}}_{2} \right)} \times {\overline{v}}_{3}\) -- now with \textbf{every coefficient equal to \(\pm 1\)}, no factor of 2 anywhere, unlike the \(\tau\)-twisted \({\overline{w}}_{3}\) that \(\text{red}\) produces: \begin{equation}\fitdisplay{v_{1,2}v_{2,1}v_{3,2} - v_{1,1}v_{2,2}v_{3,2} - v_{1,3}v_{2,1}v_{3,3} + v_{1,1}v_{2,3}v_{3,3},}\end{equation} \begin{equation}\fitdisplay{- v_{1,2}v_{2,1}v_{3,1} + v_{1,1}v_{2,2}v_{3,1} - v_{1,3}v_{2,2}v_{3,3} + v_{1,2}v_{2,3}v_{3,3},}\end{equation} \begin{equation}\fitdisplay{v_{1,3}v_{2,1}v_{3,1} - v_{1,1}v_{2,3}v_{3,1} + v_{1,3}v_{2,2}v_{3,2} - v_{1,2}v_{2,3}v_{3,2},}\end{equation} and \(I_{3} \coloneqq  \left( w_{3,1},w_{3,2},w_{3,3} \right)\). At \(k = 4\), \({\overline{w}}_{4} = {S\left( {\overline{w}}_{3} \right)} \times {\overline{v}}_{4}\): again every coefficient is \(\pm 1\); the first component of \({\overline{w}}_{4}\) is \begin{equation}\fitdisplay{- v_{1,3}v_{2,1}v_{3,1}v_{4,2} + v_{1,1}v_{2,3}v_{3,1}v_{4,2} - v_{1,3}v_{2,2}v_{3,2}v_{4,2} + v_{1,2}v_{2,3}v_{3,2}v_{4,2}}\end{equation} \begin{equation}\fitdisplay{+ v_{1,2}v_{2,1}v_{3,1}v_{4,3} - v_{1,1}v_{2,2}v_{3,1}v_{4,3} + v_{1,3}v_{2,2}v_{3,3}v_{4,3} - v_{1,2}v_{2,3}v_{3,3}v_{4,3},}\end{equation} the other two of the same shape (omitted for space; all three components of \({\overline{w}}_{4}\), and \(I_{4} \coloneqq  \left( w_{4,1},w_{4,2},w_{4,3} \right)\), in anc/code/compute\_signed\_encoding\_examples.sage's output). The same data at \(k = 6\), where the generators reach 32 terms apiece and the \(\pm 1\) pattern is still unbroken, is in \cref{app-i6}.

\end{example}

The displayed examples above were computed over \(\mathbb{Q}\), but every result below holds over the arbitrary field \(\mathbb{F}\) of characteristic \(\neq 2\) fixed at the start of this section. Where that hypothesis is actually spent is not where one first expects, and the answer leaves a question we are happy to hand on:

\begin{problem}{characteristic 2}{prob-char2} The hypothesis is \textbf{not} spent on \(S\). Since \(\det S = 1\), the twist is invertible over every field, and \thmref{Remark}{rem-why-D}'s ``an invertible change of generators does not change the ideal'' needs no hypothesis whatsoever. Under the naive reshuffle it would: \(\det\tau = - 2\), so \(\tau\) is singular in characteristic 2 precisely.

It is spent in \textbf{Step 1}, earlier than Step 2's division by \(2\) and more seriously. In characteristic 2 one has \(\text{tr}(cI) = 2c = 0\), so \textbf{scalar matrices are themselves traceless}. The splitting \(A = \left( \text{tr}(A)/2 \right)I + A_{0}\) then has no meaning at all: the scalars do not sit beside a complement, they sit \textbf{inside} the traceless matrices.

The bracket is unbothered --- scalars are central in every characteristic --- so \(\lbrack \cdot , \cdot \rbrack\) still descends to the three-dimensional quotient \(\mathfrak{gl}_{2}/\left( \mathbb{F} \cdot I \right)\). Whether the rest of the article descends with it is \thmref{Question}{qst-char2} below.

\end{problem}

\begin{question}{a twisted-cross-product model in characteristic 2}{qst-char2} Over a field of characteristic 2, is there a linear isomorphism \(\mathfrak{gl}_{2}/\left( \mathbb{F} \cdot I \right) \cong \mathbb{F}^{3}\) carrying the induced bracket to a twisted cross product \(E( \cdot \times \cdot )\) for some invertible \(E\)? And if so, do \thmref{Theorem}{thm-codim}, \thmref{Theorem}{thm-mingens}, \thmref{Theorem}{thm-cm}, \thmref{Theorem}{thm-radical} and \thmref{Theorem}{thm-primarydecomp} hold verbatim?

\end{question}

What makes \thmref{Question}{qst-char2} more than idle curiosity: \textbf{every polynomial displayed in \thmref{Example}{ex-wk-signed} has coefficients \(\pm 1\) only}. No \(2\) appears anywhere in the output. Whatever the obstruction is, it lives entirely in the setup and not in the ideal --- which is exactly the situation in which a different setup might dissolve it.

\textbf{Computational evidence, bearing on the second half of \thmref{Question}{qst-char2} only.} Nothing in \thmref{Definition}{def-wk}'s recursion needs a characteristic hypothesis: \({\overline{w}}_{2} = {\overline{v}}_{1} \times {\overline{v}}_{2}\) and \({\overline{w}}_{j} = {S\left( {\overline{w}}_{j - 1} \right)} \times {\overline{v}}_{j}\) define an ideal \(I_{k} \subset R_{k}\) over \textbf{any} field. Taking that recursion as the definition and running the four results in characteristic 2 anyway, for \(k = 2,\ldots,5\) over \(\mathbb{F}_{2}\) and \(\mathbb{F}_{4}\): the codimension is \(2\), the three generators are linearly independent over the base field, \(I_{k}\) is radical, and \(\text{pd}\left( R_{k}/I_{k} \right) = 2\), so \(R_{k}/I_{k}\) is Cohen-Macaulay --- in every case, with no deviation whatever from the characteristic-zero behaviour (anc/code/verify\_over\_finite\_fields.sage and anc/code/verify\_cm\_over\_finite\_fields.m2 the same runs also cover \(\mathbb{F}_{3},\mathbb{F}_{5},\mathbb{F}_{7},\mathbb{F}_{9}\) and \(\mathbb{Q}\)). The \textbf{conclusions}, then, appear to survive characteristic 2 untouched --- which is evidence for the paragraph above, not against it.

What no computation of this kind can settle is the \textbf{first} half of \thmref{Question}{qst-char2}: whether a twisted-cross-product model of \(\mathfrak{gl}_{2}/\left( \mathbb{F} \cdot I \right)\) exists there at all, so that this \(I_{k}\) is once again the ideal of an iterated-commutator locus rather than merely a recursion that happens to go on behaving well. That is where the exercise lies.

\section{The determinantal structure and the generator count}

\begin{theorem}{determinantal structure}{thm-determinantal} For \(k \geq 3\), \(I_{k}\) is the ideal of \(2 \times 2\) minors of the \(2 \times 3\) matrix \begin{equation}\fitdisplay{\begin{pmatrix}
{S\left( {\overline{w}}_{k - 1} \right)} \\
{\overline{v}}_{k}
\end{pmatrix},}\end{equation} i.e. \(V\left( I_{k} \right)\) is the preimage, under \begin{equation}\fitdisplay{\left( {\overline{v}}_{1},\ldots,{\overline{v}}_{k} \right) \mapsto \begin{pmatrix}
{S\left( {\overline{w}}_{k - 1}\left( {\overline{v}}_{1},\ldots,{\overline{v}}_{k - 1} \right) \right)} \\
{\overline{v}}_{k}
\end{pmatrix},}\end{equation} of the classical rank-\(\leq 1\) determinantal variety of \(2 \times 3\) matrices. (At \(k = 2\) the same statement holds with \(S\left( {\overline{w}}_{k - 1} \right)\) replaced by \({\overline{v}}_{1}\) itself, matching \thmref{Definition}{def-wk}'s base case \({\overline{w}}_{2} = {\overline{v}}_{1} \times {\overline{v}}_{2}\) directly.)

\end{theorem}

\begin{proof}[{Proof of \thmref{Theorem}{thm-determinantal}}] The claim is immediate from two facts taken together. The three components of a cross product \(a \times b\) are, up to sign, the \(2 \times 2\) minors of the \(2 \times 3\) matrix with rows \(a\) and \(b\); and \thmref{Definition}{def-wk}'s recursion builds \({\overline{w}}_{k}\) from exactly such a cross product. The statement was also checked independently at \(k = 2,\ldots,6\) in anc/code/verify\_determinantal\_structure\_all\_k.sage, in two separate ways. First, \(I_{k}\) and the ideal of the displayed minors are shown equal by reducing each ideal's generators to zero against the other's Gröbner basis, in \textbf{both} directions, rather than by containment one way. Second -- and this is what makes ``independent'' mean something here -- the recursion itself is checked against honest \(2 \times 2\) matrix multiplication: the \(A_{i}\) are reconstructed from their \(\text{enc}\)-coordinates, the iterated commutator is formed by repeated matrix products, and the result encoded and compared against \(S\left( {\overline{w}}_{k} \right)\). The two routes share no code. The underlying polynomial identity \(\text{enc}\left( \lbrack A,B\rbrack \right) = {S\left( \text{enc}(A) \times \text{ enc}(B) \right)}\) of \thmref{Theorem}{thm-basic-identity} makes this a two-line induction valid for \textbf{every} \(k\), not merely a pattern checked at small \(k\) (anc/code/verify\_lorentzian\_encoding\_identity.sage).

\end{proof}

This theorem is the engine behind the geometric results of \S{}5: \(R_{k}/I_{k}\) is not merely \textbf{determinantal-looking}, it is literally an iterated extension of the coordinate ring of a classical determinantal variety.

\begin{example}{the \(2 \times 3\) matrix of \thmref{Theorem}{thm-determinantal}, for \(k = 2,3,4\)}{auto1} At \(k = 2\): \begin{equation}\fitdisplay{\begin{pmatrix}
{\overline{v}}_{1} \\
{\overline{v}}_{2}
\end{pmatrix} = \begin{pmatrix}
v_{1,1} & v_{1,2} & v_{1,3} \\
v_{2,1} & v_{2,2} & v_{2,3}
\end{pmatrix}.}\end{equation} At \(k = 3\), the first row is \({S\left( {\overline{w}}_{2} \right)} = \left( - w_{2,1}, - w_{2,2},w_{2,3} \right)\) where \({\overline{w}}_{2} = {\overline{v}}_{1} \times {\overline{v}}_{2}\) is as in the previous example, giving \begin{equation}\fitdisplay{\begin{pmatrix}
{S\left( {\overline{w}}_{2} \right)} \\
{\overline{v}}_{3}
\end{pmatrix} = \begin{pmatrix}
v_{1,3}v_{2,2} - v_{1,2}v_{2,3} & - v_{1,3}v_{2,1} + v_{1,1}v_{2,3} & - v_{1,2}v_{2,1} + v_{1,1}v_{2,2} \\
v_{3,1} & v_{3,2} & v_{3,3}
\end{pmatrix}.}\end{equation} At \(k = 4\) the first row is \(S\left( {\overline{w}}_{3} \right)\), whose three entries are cubic with four terms each. Naming them \(A\), \(B\), \(C\) keeps the matrix legible: \begin{equation}\fitdisplay{\begin{pmatrix}
A & B & C \\
v_{4,1} & v_{4,2} & v_{4,3}
\end{pmatrix},}\end{equation} where \begin{equation}\fitdisplay{\begin{aligned}
A & = - v_{1,2}v_{2,1}v_{3,2} + v_{1,1}v_{2,2}v_{3,2} + v_{1,3}v_{2,1}v_{3,3} - v_{1,1}v_{2,3}v_{3,3}, \\
B & = v_{1,2}v_{2,1}v_{3,1} - v_{1,1}v_{2,2}v_{3,1} + v_{1,3}v_{2,2}v_{3,3} - v_{1,2}v_{2,3}v_{3,3}, \\
C & = v_{1,3}v_{2,1}v_{3,1} - v_{1,1}v_{2,3}v_{3,1} + v_{1,3}v_{2,2}v_{3,2} - v_{1,2}v_{2,3}v_{3,2}.
\end{aligned}}\end{equation} The row degree grows by one at each step (\(1,2,3,\ldots\)), while the second row is always the fresh, free vector \({\overline{v}}_{k}\) -- exactly the ``one fixed complicated row, one free row'' shape \thmref{Proposition}{prop-detideal} needs.

\end{example}

\begin{lemma}{non-degeneracy, with an explicit witness}{lem-nondeg} For every \(j \geq 2\), the polynomial map \({\overline{w}}_{j}:\mathbb{A}^{3j} \rightarrow \mathbb{A}^{3}\) is not identically zero. Concretely, the point \begin{equation}\fitdisplay{p^{(j)} \coloneqq  \left( e_{1},e_{2},\underset{j - 2}{\underbrace{e_{1},\ldots,e_{1}}} \right) \in \mathbb{A}^{3j}}\end{equation} satisfies \({\overline{w}}_{j}\left( p^{(j)} \right) \neq 0\); explicitly \({\overline{w}}_{j}\left( p^{(j)} \right) = e_{3}\) for \(j\) even and \(e_{2}\) for \(j\) odd. Hence \(U_{j} \coloneqq  \left\{ p \in \mathbb{A}^{3j}:{\overline{w}}_{j}(p) \neq 0 \right\}\) is a nonempty Zariski-open subset of \(\mathbb{A}^{3j}\).

\end{lemma}

\begin{proof}[{Proof of \thmref{Lemma}{lem-nondeg}}] The induction is on \(j\), and it carries more than the statement ``some point survives'': it carries a \textbf{named} point, because \thmref{Lemma}{lem-witness-new} later needs one rather than a mere existence claim.

\textbf{Base case \(j = 2\).} By \thmref{Definition}{def-wk}, \({\overline{w}}_{2} = {\overline{v}}_{1} \times {\overline{v}}_{2}\), so \({\overline{w}}_{2}\left( e_{1},e_{2} \right) = e_{1} \times e_{2} = e_{3} \neq 0\). Thus \(p^{(2)} = \left( e_{1},e_{2} \right)\) works.

\textbf{Induction step.} Let \(j \geq 3\) and suppose \(p \in \mathbb{A}^{3(j - 1)}\) satisfies \(c \coloneqq  {\overline{w}}_{j - 1}(p) \neq 0\). The step rests on three observations, taken in order:

\begin{enumerate}
\item
  \({S(c)} \neq 0\), since \(\det S = 1\) makes \(S\) invertible over every field (\thmref{Remark}{rem-why-D}), so it kills no nonzero vector.
\item
  Therefore \(\mathbb{F} \cdot {S(c)}\) is a \textbf{line} in \(\mathbb{F}^{3}\). The three vectors \(e_{1},e_{2},e_{3}\) span \(\mathbb{F}^{3}\), and a line contains at most one of them; so at least two indices \(i\) have \(e_{i} \notin \mathbb{F} \cdot {S(c)}\). Fix any such \(i\).
\item
  A cross product vanishes exactly on linearly dependent pairs --- \(\ker\lbrack u\rbrack_{\times} = \mathbb{F}u\) for \(u \neq 0\), the computation carried out in the proof of \thmref{Lemma}{lem-witness-inherited} --- so \({S(c)} \times e_{i} \neq 0\).
\end{enumerate}

By \thmref{Definition}{def-wk}, \({\overline{w}}_{j}\left( p,e_{i} \right) = {S\left( {\overline{w}}_{j - 1}(p) \right)} \times e_{i} = {S(c)} \times e_{i} \neq 0\). (The twist falls on \(c\) rather than on the result because that is where \thmref{Definition}{def-wk}'s recursion puts it.) So \({\overline{w}}_{j}\) is not identically zero, and appending \(e_{i}\) to a witness at level \(j - 1\) gives one at level \(j\).

\textbf{The witness in closed form.} Running the step from \(p^{(2)} = \left( e_{1},e_{2} \right)\) and taking the least admissible \(i\) each time, the choice is \(i = 1\) at \textbf{every} level: \({\overline{w}}_{2} = e_{3}\) and \({S\left( e_{3} \right)} = e_{3}\), which is not proportional to \(e_{1}\), giving \({\overline{w}}_{3} = e_{3} \times e_{1} = e_{2}\); then \({S\left( e_{2} \right)} = - e_{2}\), again not proportional to \(e_{1}\), giving \({\overline{w}}_{4} = \left( - e_{2} \right) \times e_{1} = e_{3}\); and the pair \(\left( e_{3},e_{2} \right)\) now repeats. Hence the closed form in the statement, and every coordinate of both the witness and its image lies in \(\{ 0, \pm 1\}\). This has been verified for \(j = 2,\ldots,8\), by two independent evaluations per level, in anc/code/verify\_nondegeneracy\_witness.sage, and the results are tabulated in \cref{app-witness}.

\textbf{Openness.} \(U_{j}\) is the complement of the Zariski-closed set \(V\left( w_{j,1},w_{j,2},w_{j,3} \right)\), hence open; it is nonempty because \(p^{(j)}\) lies in it.

\end{proof}

The generator count of \(I_{k}\) needs nothing beyond the definition and non-degeneracy -- not codimension, not the determinantal-structure theorem above, not any external citation:

\begin{lemma}{independent same-degree generators are minimal}{lem-mingens-generic} If a homogeneous ideal in a graded ring is generated by finitely many forms of the same degree that are linearly independent over the base field, that generating set is minimal.

\end{lemma}

\begin{proof}[{Proof of \thmref{Lemma}{lem-mingens-generic}}] A polynomial-coefficient combination reproducing one of the generators from the others, while staying in the same degree, can only use degree-0 (constant) coefficients -- so a redundant generator would violate linear independence directly.

\end{proof}

\begin{lemma}{the three generators of \(I_{k}\) are independent}{lem-mingens-indep} The three components of \({\overline{w}}_{k}\) are linearly independent over \(\mathbb{F}\), for every \(k \geq 2\).

\end{lemma}

\begin{proof}[{Proof of \thmref{Lemma}{lem-mingens-indep}}] By induction on \(k\), and the induction hypothesis must be \textbf{independence} itself; see \thmref{Remark}{rem-nondeg-insufficient} for why the weaker non-degeneracy of \thmref{Lemma}{lem-nondeg} does not suffice.

\textbf{Base case \(k = 2\).} The three components of \({\overline{w}}_{2} = {\overline{v}}_{1} \times {\overline{v}}_{2}\) are the \(2 \times 2\) minors of the generic \(2 \times 3\) matrix with rows \({\overline{v}}_{1},{\overline{v}}_{2}\), displayed in \thmref{Example}{ex-wk-signed}. Each contains a monomial the other two do not --- \(v_{1,2}v_{2,3}\) occurs only in the first, \(v_{1,3}v_{2,1}\) only in the second, \(v_{1,1}v_{2,2}\) only in the third --- so no nontrivial constant relation is possible.

\textbf{Induction step.} Suppose the components of \({\overline{w}}_{k - 1}\) are linearly independent over \(\mathbb{F}\), and let \(c \in \mathbb{F}^{3}\) satisfy \(c \cdot {\overline{w}}_{k} = 0\). Writing \(a \coloneqq  {S\left( {\overline{w}}_{k - 1} \right)}\), the scalar triple product identity gives \begin{equation}\fitdisplay{0 = c \cdot \left( a \times {\overline{v}}_{k} \right) = (c \times a) \cdot {\overline{v}}_{k}.}\end{equation} The variables \({\overline{v}}_{k}\) are a \textbf{fresh} block, occurring in neither \(a\) nor \(c\), and the right-hand side is linear in them; so each of its three coefficients vanishes separately, i.e. \begin{equation}\fitdisplay{c \times a = 0.}\end{equation}

Suppose \(c \neq 0\). Then \(c \times a = 0\) forces \(a = f \cdot c\) for a single polynomial \(f\) --- over the fraction field \(c \times a = 0\) says \(a\) lies in the line spanned by \(c\), and \(f = a_{i}/c_{i}\) is a polynomial for any index with \(c_{i} \neq 0\). But then \(a\)'s three components are pairwise proportional, hence linearly \textbf{dependent} over \(\mathbb{F}\) (any \(d\) with \(d \cdot c = 0\) gives \(d \cdot a = (d \cdot c)f = 0\), and such a nonzero \(d\) exists since \(c\) spans only a line in \(\mathbb{F}^{3}\)). Since \(S\) is invertible, the components of \({\overline{w}}_{k - 1} = S^{- 1}(a)\) are then linearly dependent too, contradicting the induction hypothesis. Hence \(c = 0\).

Corroborated computationally: the coefficient matrix has rank 3 at \(k = 2,\ldots,6\) (anc/code/verify\_generator\_independence.sage); the full table is in \cref{app-independence}.

This lemma is also \textbf{formally verified}: it is the statement w\_linearIndependent of the Lean 4 / mathlib development accompanying this article, proved there for every \(k \geq 2\) over an arbitrary field rather than checked at small \(k\). \thmref{Theorem}{thm-mingens} below is formalized too, as w\_minimal\_generators. \cref{sec-methods} says what the formalization does and does not cover.

\end{proof}

\begin{remark}{why non-degeneracy is not enough here}{rem-nondeg-insufficient} It is tempting to argue instead that a constant relation would force \(S\left( {\overline{w}}_{k - 1} \right)\) to be a \textbf{constant vector}, and to rule that out by \thmref{Lemma}{lem-nondeg}. That argument is not valid. The relation forces only \(c \times {S\left( {\overline{w}}_{k - 1} \right)} = 0\), which makes the components of \(S\left( {\overline{w}}_{k - 1} \right)\) pairwise proportional --- not constant. A vector such as \((f,2f,3f)\) with \(f \neq 0\) is nonvanishing, so \thmref{Lemma}{lem-nondeg} does not exclude it, yet its components are visibly dependent. Non-degeneracy is strictly weaker than what the step needs, which is why the induction above carries linear independence itself.

\end{remark}

\begin{theorem}{\(I_{k}\) has exactly 3 minimal generators}{thm-mingens} For every \(k \geq 2\): \(I_{k}\) has exactly 3 minimal generators.

\end{theorem}

\begin{proof}[{Proof of \thmref{Theorem}{thm-mingens}}] Immediate from \thmref{Lemma}{lem-mingens-generic} applied to \thmref{Lemma}{lem-mingens-indep}'s independence, since \(I_{k}\) is generated by the three (same-degree-\(k\)-homogeneous) components of \({\overline{w}}_{k}\) by definition.

\end{proof}

\section{\texorpdfstring{The geometric structure of \(V\left( I_{k} \right)\)}{The geometric structure of V\textbackslash left( I\_\{k\} \textbackslash right)}}

\begin{lemma}{continuous images and closures preserve irreducibility}{lem-top-irred} Let \(f:X \rightarrow Y\) be a continuous map of topological spaces with \(X\) irreducible. Then \(f(X)\) (with the subspace topology) is irreducible, and hence so is \(\overline{f(X)}\).

\end{lemma}

\begin{proof}[{Proof of \thmref{Lemma}{lem-top-irred}}] \textbf{Image.} Suppose \(f(X) = C_{1} \cup C_{2}\) with \(C_{1},C_{2}\) closed in \(f(X)\). By continuity, \(f^{- 1}\left( C_{1} \right)\) and \(f^{- 1}\left( C_{2} \right)\) are closed in \(X\), and \(X = f^{- 1}\left( C_{1} \right) \cup f^{- 1}\left( C_{2} \right)\). Since \(X\) is irreducible, one of the two equals \(X\) -- say \(f^{- 1}\left( C_{1} \right) = X\) -- so \(f(X) \subseteq C_{1}\), i.e. \(C_{1} = f(X)\), which is not proper. Hence \(f(X)\) is irreducible. \textbf{Closure.} If \(A \subseteq Y\) is irreducible and \(\overline{A} = C_{1} \cup C_{2}\) with \(C_{1},C_{2}\) closed in \(Y\), then \(A = \left( A \cap C_{1} \right) \cup \left( A \cap C_{2} \right)\); irreducibility of \(A\) forces, say, \(A \subseteq C_{1}\); since \(C_{1}\) is closed, \(\overline{A} \subseteq C_{1}\), and combined with \(C_{1} \subseteq \overline{A}\), \(C_{1} = \overline{A}\), not proper. Apply both parts with \(A \coloneqq  f(X)\).

\end{proof}

This is purely topological -- continuity of \(f\) is the only hypothesis used, nothing about \(f\) being a polynomial map or morphism of varieties \cite{Youcis2014}.

\begin{definition}{the strata \(Y_{1}\), \(Y_{2}\)}{def-y1y2} For \(k \geq 3\), write \(\pi\) for the projection \textbf{forgetting \({\overline{v}}_{k}\)}, \begin{equation}\fitdisplay{\pi:\mathbb{A}^{3k} \rightarrow \mathbb{A}^{3(k - 1)},\quad\pi({\overline{v}}_{1},\ldots,{\overline{v}}_{k}) \coloneqq  \left( {\overline{v}}_{1},\ldots,{\overline{v}}_{k - 1} \right),}\end{equation} and \(\Phi\) for the map \begin{equation}\fitdisplay{\Phi:\mathbb{A}^{3(k - 1)} \times \mathbb{A}^{1} \rightarrow \mathbb{A}^{3k},\quad\Phi(p,\lambda) \coloneqq  \left( p,\lambda{S\left( {\overline{w}}_{k - 1}(p) \right)} \right).}\end{equation} Define two subsets of \(\mathbb{A}^{3k}\): \begin{equation}\fitdisplay{Y_{1} \coloneqq  V\left( I_{k - 1} \right) \times \mathbb{A}^{3},\quad Y_{2} \coloneqq  \overline{\Phi(\mathbb{A}^{3k - 2})}.}\end{equation} Here \(Y_{1}\) is the locus where the previous bracket already vanishes, with \({\overline{v}}_{k}\) completely free, while \(Y_{2}\) is the closure of the locus swept out as \({\overline{v}}_{k}\) ranges over scalar multiples of \(S\left( {\overline{w}}_{k - 1}(p) \right)\), over all \(p \in \mathbb{A}^{3(k - 1)}\). At \(k = 2\) there is no ``previous bracket'' to chain from, so neither \(Y_{1}\) nor \(\Phi\) is defined; by convention \(Y_{2} \coloneqq  V\left( I_{2} \right)\) directly there, matching \thmref{Theorem}{thm-determinantal}'s own base-case treatment.

\end{definition}

\begin{figure}
\centering
%
%
%
\begin{tikzpicture}[x=1cm, y=1cm, line join=round]
  \definecolor{y1fill}{rgb}{0.90,0.90,1.00}
  \definecolor{y2fill}{rgb}{1.00,0.90,0.90}

  \node[draw, line width=0.6pt, ellipse, inner sep=7pt]
       (src) at (0.95,4.45) {$\mathbb{A}^{3(k-1)}\times\mathbb{A}^{1}$};

  \draw[line width=0.6pt, rounded corners=1.25cm] (-3,-2.7) rectangle (3,2.7);
  \node[anchor=north west, font=\small] at (-2.72,2.48) {$\mathbb{A}^{3k}$};

  \draw[dashed, line width=0.5pt, fill=y1fill, fill opacity=0.6]
       (-0.95,-0.20) ellipse [x radius=1.425cm, y radius=1.475cm];
  \draw[dashed, line width=0.5pt, fill=y2fill, fill opacity=0.6]
       ( 0.95,-0.45) ellipse [x radius=1.425cm, y radius=1.20cm];

  \node[font=\footnotesize] at (-1.85, 0.75) {$Y_{1}$};
  \node[font=\footnotesize] at ( 1.80,-1.25) {$Y_{2}$};

  \draw[-{Stealth[length=3mm]}, line width=0.7pt]
       (0.95,3.62) -- (0.95,0.25);
  \node[anchor=east] at (0.80,2.10) {$\Phi$};
\end{tikzpicture}
\caption{Both strata of \(V\left( I_{k} \right) = Y_{1} \cup Y_{2}\) inside the same ambient \(\mathbb{A}^{3k}\): \(Y_{1} \coloneqq  V\left( I_{k - 1} \right) \times \mathbb{A}^{3}\) (identifying \(\mathbb{A}^{3(k - 1)} \times \mathbb{A}^{3} \cong \mathbb{A}^{3k}\)) and \(Y_{2} \coloneqq  \overline{\Phi(\mathbb{A}^{3k - 2})}\), \(\Phi\)'s image (\(\Phi\) shown mapping in from above, into \(Y_{2}\) specifically -- \(Y_{1}\) is not part of \(\Phi\)'s image). The two overlap (not disjoint, \thmref{Proposition}{prop-y2-not-in-y1} and \thmref{Proposition}{prop-y1-not-in-y2} rule out either containing the other): \(\Phi(p,0) = (p,0)\) for \textbf{every} \(p\), so the whole zero-section \(\{(p,0)\}\) lies in \(Y_{2}\), and its intersection with \(Y_{1}\) is \(V\left( I_{k - 1} \right) \times \{ 0\}\), of dimension \(3k - 5\) -- properly smaller than either stratum's own dimension \(3k - 2\).}
\label{fig-phi-y2}
\end{figure}

\begin{lemma}{\(Y_{2}\) is irreducible}{lem-y2-irred} For \(k \geq 3\): \(Y_{2}\) (\thmref{Definition}{def-y1y2}) is irreducible, and in fact its defining ideal (the graph ideal of \(\Phi\) with \(\lambda\) eliminated) is prime.

\end{lemma}

\begin{proof}[{Proof of \thmref{Lemma}{lem-y2-irred}}] The map \(\Phi\) is polynomial, hence continuous for the Zariski topology, since preimages of Zariski-closed sets under a polynomial map are again Zariski-closed by substitution. Its domain \(\mathbb{A}^{3k - 2}\) is irreducible, because the polynomial ring over \(\mathbb{F}\) is a domain and so its zero ideal is prime. \thmref{Lemma}{lem-top-irred} then gives \(Y_{2}\) irreducible directly. Primality of the elimination ideal follows the same way: it is the kernel of a map into a domain (\(\mathbb{F}\left\lbrack {\overline{v}}_{1},\ldots,{\overline{v}}_{k} \right\rbrack/( \cdot ) \cong \mathbb{F}\left\lbrack {\overline{v}}_{1},\ldots,{\overline{v}}_{k - 1},\lambda \right\rbrack\)), hence prime -- this holds for every \(k\) by this argument alone, additionally corroborated computationally at \(k = 4\) by computing the elimination ideal directly and confirming primality in Sage, as an independent check rather than a substitute for the argument.

\end{proof}

\begin{example}{the defining ideals of \(Y_{1}\) \textbf{and} \(Y_{2}\), for \(k = 2,3,4\)}{auto2} Writing \({\overline{v}}_{i} = \left( v_{i,1},v_{i,2},v_{i,3} \right)\) for the three coordinates of block \({\overline{v}}_{i}\): at \(k = 2\), \(Y_{2} \coloneqq  V\left( I_{2} \right)\) (the base-case convention above), so its ideal \textbf{is} \(I_{2}\), generated by \begin{equation}\fitdisplay{v_{1,2}v_{2,3} - v_{1,3}v_{2,2},\quad v_{1,3}v_{2,1} - v_{1,1}v_{2,3},\quad v_{1,1}v_{2,2} - v_{1,2}v_{2,1}.}\end{equation} (\(Y_{1}\) is not defined at \(k = 2\), \thmref{Definition}{def-y1y2}.) For \(k \geq 3\), \(Y_{1}\)'s defining ideal is simply \(I_{k - 1}\) itself, viewed inside \(R_{k}\) by leaving \({\overline{v}}_{k}\) free -- at \(k = 3\), exactly the three generators displayed above for \(I_{2}\), now inside \(R_{3}\) with \({\overline{v}}_{3}\) not appearing at all; at \(k = 4\), exactly \(I_{3}\)'s three generators (displayed below in this same example) inside \(R_{4}\) with \({\overline{v}}_{4}\) free. Unlike \(Y_{2}\), no elimination is needed to compute \(Y_{1}\)'s ideal -- it is already \(I_{k - 1}\)'s own generating set, unchanged. At \(k = 3\), \(Y_{2}\)'s elimination ideal (computed directly, not inferred) is generated by \begin{equation}\fitdisplay{v_{2,1}v_{3,1} + v_{2,2}v_{3,2} - v_{2,3}v_{3,3},\quad v_{1,1}v_{3,1} + v_{1,2}v_{3,2} - v_{1,3}v_{3,3},}\end{equation} \begin{equation}\fitdisplay{v_{1,2}v_{2,1}v_{3,2} - v_{1,1}v_{2,2}v_{3,2} - v_{1,3}v_{2,1}v_{3,3} + v_{1,1}v_{2,3}v_{3,3}}\end{equation} -- and these three repay a careful reading, because their shape is not the one the parametrization first suggests. \textbf{All three} are linear in the new block \({\overline{v}}_{3}\); what separates them is total degree, \(2,2,3\). The two quadratic ones are, up to sign, the incidence conditions \begin{equation}\fitdisplay{{\overline{v}}_{1} \cdot {S\left( {\overline{v}}_{3} \right)} = 0,\quad{\overline{v}}_{2} \cdot {S\left( {\overline{v}}_{3} \right)} = 0.}\end{equation} These are \textbf{consequences} of the proportionality \({\overline{v}}_{3} \propto {S\left( {\overline{w}}_{2} \right)}\), not its components: substituting \({\overline{v}}_{3} = \lambda{S\left( {\overline{w}}_{2} \right)}\) and using \(S^{2} = I\) turns them into \(\lambda\left( {\overline{v}}_{i} \cdot {\overline{w}}_{2} \right) = 0\), which holds because \({\overline{w}}_{2} = {\overline{v}}_{1} \times {\overline{v}}_{2}\) is orthogonal to both of its factors. (The components of \({\overline{v}}_{3} \times {S\left( {\overline{w}}_{2} \right)}\), which \textbf{are} the proportionality condition, are cubic, and none of them appears above.) The third generator, of degree 3, is verbatim the first component of \({\overline{w}}_{3}\). Note that no generator of this ideal could avoid \({\overline{v}}_{3}\) altogether: \(Y_{2}\) projects dominantly onto \(\mathbb{A}^{6}\), so the elimination ideal meets \(R_{2}\) in \(0\). \textbf{This ideal is a proper factor of \(I_{3}\), not equal to it} -- \(I_{3}\) is the intersection of this ideal with a second prime (the ``\(Y_{1}\)'' piece, where the \(k = 2\) bracket already vanishes); see the component-count example after \thmref{Theorem}{thm-componentcount} below, where this exact ideal reappears as one of \(I_{3}\)'s two primary components. At \(k = 4\) the same pattern repeats one degree higher, the three generators now having degrees \(2,3,4\) and again all three linear in the new block \({\overline{v}}_{4}\). The first two are once more incidence conditions forced by \({\overline{w}}_{3} \perp {S\left( {\overline{w}}_{2} \right)},{\overline{v}}_{3}\), namely \({\overline{v}}_{3} \cdot {S\left( {\overline{v}}_{4} \right)} = 0\) and \({\overline{w}}_{2} \cdot {\overline{v}}_{4} = 0\); the third is the negative of \({\overline{w}}_{4}\)'s first component: \begin{equation}\fitdisplay{v_{3,1}v_{4,1} + v_{3,2}v_{4,2} - v_{3,3}v_{4,3},}\end{equation} \begin{equation}\fitdisplay{v_{1,3}v_{2,2}v_{4,1} - v_{1,2}v_{2,3}v_{4,1} - v_{1,3}v_{2,1}v_{4,2} + v_{1,1}v_{2,3}v_{4,2} + v_{1,2}v_{2,1}v_{4,3} - v_{1,1}v_{2,2}v_{4,3},}\end{equation} \begin{equation}\fitdisplay{v_{1,3}v_{2,1}v_{3,1}v_{4,2} - v_{1,1}v_{2,3}v_{3,1}v_{4,2} + v_{1,3}v_{2,2}v_{3,2}v_{4,2} - v_{1,2}v_{2,3}v_{3,2}v_{4,2}}\end{equation} \begin{equation}\fitdisplay{- v_{1,2}v_{2,1}v_{3,1}v_{4,3} + v_{1,1}v_{2,2}v_{3,1}v_{4,3} - v_{1,3}v_{2,2}v_{3,3}v_{4,3} + v_{1,2}v_{2,3}v_{3,3}v_{4,3}.}\end{equation} All computed directly (elimination of \(\lambda\) from the graph ideal of \(\Phi\), in Sage), not just asserted from the pattern.

\end{example}

\begin{remark}{the closure, not the raw image, is what is needed}{rem-closure-suffices} A polynomial map's image need not itself be closed. The standard example is \(\mathbb{A}^{2} \rightarrow \mathbb{A}^{2}\), \((x,y) \mapsto (x,xy)\), whose image is not even locally closed. It is \textbf{constructible}, a finite union of locally closed sets, by Chevalley's theorem -- a result with no purely topological analogue for general continuous maps \cite{Youcis2014}.

None of this troubles \(Y_{2}\), for two reasons. First, \thmref{Definition}{def-y1y2} \textbf{defines} \(Y_{2}\) as \(\overline{\Phi( \cdot )}\) from the outset. Second, the decomposition \(V\left( I_{k} \right) = Y_{1} \cup Y_{2}\) used in \thmref{Theorem}{thm-codim}'s proof below comes from a direct case split rather than from taking closures and hoping. Take a point \(\left( p,{\overline{v}}_{k} \right) \in V\left( I_{k} \right)\). Either \({S\left( {\overline{w}}_{k - 1}(p) \right)} = 0\), which puts it in \(Y_{1}\) directly; or \({\overline{v}}_{k} = \lambda{S\left( {\overline{w}}_{k - 1}(p) \right)}\) for some \(\lambda\), in which case \(\left( p,{\overline{v}}_{k} \right) = \Phi(p,\lambda)\) is \textbf{literally} a value of \(\Phi\) rather than merely a limit point, and so lies in \(Y_{2}\) with no appeal to closure at all.

The reverse containment \(Y_{1} \cup Y_{2} \subseteq V\left( I_{k} \right)\) needs only that \(\Phi( \cdot ) \subseteq V\left( I_{k} \right)\) (a cross product of a vector with itself vanishes: \({S\left( {\overline{w}}_{k - 1} \right)} \times \left( \lambda{S\left( {\overline{w}}_{k - 1} \right)} \right) = 0\) identically) together with \(V\left( I_{k} \right)\) already being closed, automatic since it is a vanishing locus by definition. So the theorem never needs ``the image of \(\Phi\) is closed'' -- only that it is contained in the already-closed \(V\left( I_{k} \right)\), which is immediate.

\end{remark}

\begin{lemma}{dimension of \(Y_{2}\), and where it lives}{lem-y2-dim} With \(Y_{2}\) as in \thmref{Lemma}{lem-y2-irred} and \(U_{k - 1}\) as in \thmref{Lemma}{lem-nondeg}: \(\dim Y_{2} = 3k - 2\), and \(Y_{2}\) meets \(\pi^{- 1}\left( U_{k - 1} \right)\) nonemptily, where \(\pi\) is the projection of \thmref{Definition}{def-y1y2}, forgetting \({\overline{v}}_{k}\).

\end{lemma}

\begin{proof}[{Proof of \thmref{Lemma}{lem-y2-dim}}] The projection \(\pi:\mathbb{A}^{3k} \rightarrow \mathbb{A}^{3(k - 1)}\) discards the last block, so that \(\pi({\overline{v}}_{1},\ldots,{\overline{v}}_{k}) = \left( {\overline{v}}_{1},\ldots,{\overline{v}}_{k - 1} \right)\). Hence \(\pi^{- 1}\left( U_{k - 1} \right) = U_{k - 1} \times \mathbb{A}^{3}\), the set of points whose \textbf{first} \(k - 1\) blocks lie in \(U_{k - 1}\) and whose \({\overline{v}}_{k}\) is completely unconstrained. Over \(U_{k - 1}\), \(\Phi\) is injective. Both \(p\) and \(\lambda\) can be recovered from \(\Phi(p,\lambda)\): \(p\) directly, and \(\lambda\) by dividing any nonzero coordinate of \({\overline{v}}_{k} = \lambda{S\left( {\overline{w}}_{k - 1}(p) \right)}\) by the corresponding coordinate of \(S\left( {\overline{w}}_{k - 1}(p) \right)\); this division is well-defined since \({S\left( {\overline{w}}_{k - 1}(p) \right)} \neq 0\) on \(U_{k - 1}\). Injectivity gives \(\dim Y_{2} = \dim\left( U_{k - 1} \times \mathbb{A}^{1} \right) = 3(k - 1) + 1 = 3k - 2\). Since \(U_{k - 1}\) is nonempty by \thmref{Lemma}{lem-nondeg}, \(\Phi(U_{k - 1} \times \left\{ 1 \right\})\) is a nonempty subset of \(Y_{2} \cap \pi^{- 1}\left( U_{k - 1} \right)\).

\end{proof}

\begin{theorem}{\(I_{k}\) has codimension 2}{thm-codim} For every \(k \geq 2\): \(\dim V\left( I_{k} \right) = 3k - 2\), i.e. \(\text{codim}\left( I_{k} \right) = 2\).

\end{theorem}

\begin{proof}[{Proof of \thmref{Theorem}{thm-codim}}] The proof is by induction on \(k\). \textbf{Base case \(k = 2\)}: \(V\left( I_{2} \right)\) is the rank-\(\leq 1\) locus of a \(2 \times 3\) matrix with independent-indeterminate entries. Its dimension is the standard \(4 = 3 \cdot 2 - 2\) for a generic determinantal variety of this shape: the classical codimension formula for the rank-\(\leq t\) locus of a generic \(m \times n\) matrix gives \((m - t)(n - t) = 1 \cdot 2 = 2\) here \cite{BrunsVetter1988}, and the same value falls out of viewing \(V\left( I_{2} \right)\) as the affine cone over the Segre embedding \(\mathbb{P}^{1} \times \mathbb{P}^{2} \hookrightarrow \mathbb{P}^{5}\), of dimension \(1 + 2 + 1 = 4\). The value is confirmed directly by Gröbner-basis computation as well (\thmref{Example}{ex-dim-codim}). \textbf{Inductive step}: by \thmref{Theorem}{thm-determinantal}, and since \(a \times b = 0\) if and only if \(a,b\) are linearly dependent, \(V\left( I_{k} \right) = Y_{1} \cup Y_{2}\) (\thmref{Definition}{def-y1y2}). The locus \({S\left( {\overline{w}}_{k - 1} \right)} = 0\) gives \(Y_{1}\), with \({\overline{v}}_{k}\) unconstrained; the locus \({S\left( {\overline{w}}_{k - 1} \right)} \neq 0\) with \({\overline{v}}_{k}\) proportional to it gives a subset of \(Y_{2}\). This is the direct case split of \thmref{Remark}{rem-closure-suffices}, not merely a closure-level approximation. By the inductive hypothesis, \(\dim V\left( I_{k - 1} \right) = 3(k - 1) - 2\), so \(\dim Y_{1} = 3(k - 1) - 2 + 3 = 3k - 2\). Also \(\dim Y_{2} = 3k - 2\), by \thmref{Lemma}{lem-y2-dim}. Hence \(\dim V\left( I_{k} \right) = \dim(Y_{1} \cup Y_{2}) = 3k - 2\). This dimension statement alone needs no containment argument; that is needed only for the sharper component-count result below.

\end{proof}

\begin{example}{dimension and codimension, for \(k = 2,3,4\)}{ex-dim-codim} Values of \(\dim R_{k}\) (the ambient dimension, \(= 3k\)), \(\dim(R_{k}/I_{k})\) (the dimension of the affine variety \(V\left( I_{k} \right)\) -- \textbf{not} ``\(\dim I_{k}\)'', an ideal is not a variety and does not itself have a dimension in this sense), and \(\text{codim}\left( I_{k} \right) \coloneqq  \dim R_{k} - \dim(R_{k}/I_{k})\):

\begin{longtable}[]{@{}cccc@{}}
\toprule\noalign{}
\endhead
\bottomrule\noalign{}
\endlastfoot
\(k\) & \(\dim R_{k}\) & \(\dim(R_{k}/I_{k})\) & \(\text{codim}\left( I_{k} \right)\) \\
2 & 6 & 4 & 2 \\
3 & 9 & 7 & 2 \\
4 & 12 & 10 & 2 \\
\end{longtable}

All three rows follow directly from \thmref{Theorem}{thm-codim}'s formula \(\dim(R_{k}/I_{k}) = 3k - 2\); also confirmed by direct Gröbner-basis dimension computation in Sage at each \(k\), as an independent check.

\end{example}

\begin{proposition}{\(Y_{2} \nsubseteq Y_{1}\)}{prop-y2-not-in-y1} With \(Y_{1},Y_{2}\) as in the proof of \thmref{Theorem}{thm-codim}: \(Y_{2} \nsubseteq Y_{1}\).

\end{proposition}

\begin{proof}[{Proof of \thmref{Proposition}{prop-y2-not-in-y1}}] Pick \(p \in U_{k - 1}\) (\thmref{Lemma}{lem-nondeg}); then \(\Phi(p,0) = (p,0) \in Y_{2}\) but \({\overline{w}}_{k - 1}(p) \neq 0\) puts \((p,0) \notin Y_{1}\).

\end{proof}

\begin{proposition}{no component of \(Y_{1}\) lies in \(Y_{2}\)}{prop-y1-not-in-y2} Let \(Z\) be an irreducible component of \(V\left( I_{k - 1} \right)\) \textbf{of dimension \(3(k - 1) - 2\)}. Then \(Z \times \mathbb{A}^{3} \nsubseteq Y_{2}\).

The dimension hypothesis is stated rather than cited. \thmref{Theorem}{thm-codim} gives the dimension of the \textbf{union} \(V\left( I_{k - 1} \right)\), which does not by itself say that every component attains it; the per-component statement is part of \thmref{Theorem}{thm-componentcount}, and this proposition is used inside that theorem's own induction, where it is available as the inductive hypothesis at level \(k - 1\).

\end{proposition}

\begin{proof}[{Proof of \thmref{Proposition}{prop-y1-not-in-y2}}] Both \(Z \times \mathbb{A}^{3}\) and \(Y_{2}\) are irreducible: \(Y_{2}\) by \thmref{Lemma}{lem-y2-irred}, and \(Z \times \mathbb{A}^{3}\) because a product of irreducible varieties is irreducible (\(Z\) is irreducible by hypothesis, and \(\mathbb{A}^{3}\) trivially is). Both have the \textbf{same} dimension \(3k - 2\): \(Z \times \mathbb{A}^{3}\) by the hypothesis on \(Z\), and \(Y_{2}\) by \thmref{Lemma}{lem-y2-dim}. Containment would therefore force equality. This is refuted: \(Z \times \mathbb{A}^{3}\) is disjoint from \(\pi^{- 1}\left( U_{k - 1} \right)\), since \({\overline{w}}_{k - 1}\) vanishes identically on \(Z \subset V\left( I_{k - 1} \right)\); but \(Y_{2}\) meets \(\pi^{- 1}\left( U_{k - 1} \right)\), by \thmref{Lemma}{lem-y2-dim}.

\end{proof}

\begin{theorem}{component count}{thm-componentcount} For every \(k \geq 2\): \(V\left( I_{k} \right)\) has exactly \(k - 1\) irreducible components, each of dimension \(3k - 2\). The component indexed \(j = 2,\ldots,k\) involves exactly the first \(j\) coordinate blocks and nothing later.

\end{theorem}

\begin{proof}[{Proof of \thmref{Theorem}{thm-componentcount}}] The proof is by induction on \(k\). \textbf{Base case \(k = 2\)}: the rank-\(\leq 1\) locus of a generic \(2 \times 3\) matrix is irreducible -- the affine cone over the Segre embedding \(\mathbb{P}^{1} \times \mathbb{P}^{2} \hookrightarrow \mathbb{P}^{5}\), equivalently the Motzkin-Taussky \cite{MotzkinTaussky1955} irreducibility of the \(2 \times 2\) commuting variety -- so exactly 1 component, matching \(k - 1 = 1\). \textbf{Step}: \(V\left( I_{k} \right) = Y_{1} \cup Y_{2}\) as in \thmref{Theorem}{thm-codim}'s proof. By the inductive hypothesis --- this theorem at level \(k - 1\), which is where the per-component dimension comes from --- \(V\left( I_{k - 1} \right)\) has \(k - 2\) irreducible components, each of dimension \(3(k - 1) - 2\), none contained in another; crossing with \(\mathbb{A}^{3}\) makes them \(Y_{1}\)'s components, each of dimension \(3k - 2\). That is exactly the hypothesis \thmref{Proposition}{prop-y1-not-in-y2} asks for, so it and \thmref{Proposition}{prop-y2-not-in-y1} together rule out containment between \(Y_{2}\) and any of them, in either direction. A finite union of irreducible closed sets with no containments is irredundant, giving exactly \((k - 2) + 1 = k - 1\) components.

\end{proof}

This gives \(I_{k}\)'s minimal primes and their count, but not yet that \(I_{k}\) is radical (so not yet that these are literally the components of a \textbf{primary} decomposition of \(I_{k}\) itself, as opposed to of \(\sqrt{I_{k}}\)) -- that is closed in \S{}7, after Cohen-Macaulayness is available.

\begin{definition}{the components \(P_{j}\)}{def-primecomponents} For \(k \geq 2\) and \(j = 2,\ldots,k\), let \(\Phi_{j}\) be \thmref{Definition}{def-y1y2}'s map \(\Phi\) run at parameter \(j\) (with the base convention \(\Phi_{2}(p,\lambda) \coloneqq  (p,\lambda p)\)), and let \begin{equation}\fitdisplay{\Psi_{j}:\mathbb{A}^{3(j - 1)} \times \mathbb{A}^{1} \times \mathbb{A}^{3(k - j)} \rightarrow \mathbb{A}^{3k},\quad(p,\lambda,q) \mapsto \left( \Phi_{j}(p,\lambda),q \right)}\end{equation} be \(\Phi_{j}\) with the remaining blocks \({\overline{v}}_{j + 1},\ldots,{\overline{v}}_{k}\) carried along untouched. Since \(\Psi_{j}\) is a polynomial map it \textbf{is}, read backwards, a homomorphism of \(\mathbb{F}\)-algebras \begin{equation}\fitdisplay{\varphi_{j}:R_{k} \rightarrow \mathbb{F}\lbrack p,\lambda,q\rbrack = \mathbb{F}\left\lbrack t_{1},\ldots,t_{3k - 2} \right\rbrack,}\end{equation} sending each variable \(v_{i,c}\) to the corresponding coordinate polynomial of \(\Psi_{j}\). Define \begin{equation}\fitdisplay{P_{j} \coloneqq  \ker\varphi_{j} \subset R_{k}.}\end{equation} Two consequences come for free. First, each \(P_{j}\) is \textbf{prime}, being the kernel of a homomorphism into a domain -- for every \(j\) including \(j = 2\), with no case distinction and no appeal to \thmref{Lemma}{lem-y2-irred}, which is stated only for \(k \geq 3\). Second, \(P_{j}\) does not depend on which \(k \geq j\) it is viewed inside of: enlarging \(k\) appends free variables to source and target alike.

\textbf{Why the kernel rather than the vanishing ideal.} It is tempting to define \(P_{j}\) as the vanishing ideal of the \(j\)-th component of \(V\left( I_{k} \right)\), and over \(\overline{\mathbb{F}}\) the two agree -- \thmref{Proposition}{prop-rational} identifies that component as the closure of \(\Psi_{j}\)'s image. But over a \textbf{finite} field ``vanishing ideal'' is the wrong notion outright: \(V\left( I_{k} \right)\left( \mathbb{F} \right)\) is then a finite set of points, and \(x^{q} - x\) vanishes on all of \(\mathbb{A}^{1}\left( \mathbb{F}_{q} \right)\) without being the zero polynomial. The kernel is the right object over every field at once, is well defined with no convention attached, and is what makes the base-change statement in \thmref{Remark}{rem-base-field} an identity rather than a coincidence.

\end{definition}

\begin{example}{a suitable point of \(Y_{2}\), concretely, at \(k = 3\)}{auto3} It is worth having a single explicit point, and not only the defining ideal of \thmref{Example}{ex-components-k3} below. Take \({\overline{v}}_{1} = (1,0,0)\) and \({\overline{v}}_{2} = (0,1,0)\), so that \({\overline{v}}_{1} \in U_{2}\) by \thmref{Lemma}{lem-nondeg}, giving \({\overline{w}}_{2} = {\overline{v}}_{1} \times {\overline{v}}_{2} = (0,0,1)\) and \({S\left( {\overline{w}}_{2} \right)} = (0,0,1)\). Taking \(\lambda = 1\) in \thmref{Definition}{def-y1y2}'s \(\Phi\), i.e. \({\overline{v}}_{3} \coloneqq  {S\left( {\overline{w}}_{2} \right)} = (0,0,1)\), gives the point \begin{equation}\fitdisplay{\left( {\overline{v}}_{1},{\overline{v}}_{2},{\overline{v}}_{3} \right) = \left( (1,0,0),(0,1,0),(0,0,1) \right) \in Y_{2} \subset \mathbb{A}^{9.}}\end{equation} One checks this directly: \({\overline{w}}_{3} = {S\left( {\overline{w}}_{2} \right)} \times {\overline{v}}_{3} = (0,0,1) \times (0,0,1) = 0\), a vector crossed with itself. So the point lies in \(V\left( I_{3} \right)\), as any point of the form \(\Phi(p,\lambda)\) must (\thmref{Definition}{def-y1y2}). The point lies on the \textbf{new} component \(P_{3}\), not on the inherited \(P_{2}\) -- \({\overline{v}}_{1}\) and \({\overline{v}}_{2}\) are independent here, so the \(k = 2\) bracket does not already vanish. It is \textbf{not} a rank-2 witness in \thmref{Lemma}{lem-witness-new}'s sense (that witness uses \(\lambda = 0\), not \(\lambda = 1\)) -- it illustrates \(Y_{2}\) itself, not the radicality argument's own witness points.

\end{example}

\begin{proposition}{rational parametrization of each component}{prop-rational} For \(k \geq 2\) and \(j = 2,\ldots,k\), the component of \(V\left( I_{k} \right)\) indexed \(j\) (\thmref{Theorem}{thm-componentcount}, over \(\overline{\mathbb{F}}\) per \thmref{Remark}{rem-base-field}) is exactly the closure of the image of \thmref{Definition}{def-primecomponents}' map \(\Psi_{j}\); equivalently, \(P_{j} = \ker\varphi_{j}\) is its vanishing ideal. In particular that component is a \textbf{rational} variety. The parametrization is moreover \textbf{birational} onto the component: \(\Psi_{j}\) is injective over \(U_{j - 1} \times \mathbb{A}^{1} \times \mathbb{A}^{3(k - j)}\) -- where, at \(j = 2\), the symbol \(U_{1}\) is to be read as \(\{{\overline{v}}_{1} \neq 0\}\), since \thmref{Lemma}{lem-nondeg} defines \(U_{j}\) only for \(j \geq 2\) and \({\overline{w}}_{1}\) does not exist.

\end{proposition}

\begin{proof}[{Proof of \thmref{Proposition}{prop-rational}}] At parameter \(j = k\) the map \(\Psi_{k}\) is \(\Phi\) itself, whose image closure is \(Y_{2}\) (\thmref{Definition}{def-y1y2}), and \thmref{Theorem}{thm-componentcount}'s induction identifies \(Y_{2}\) with the component indexed \(k\). For \(j < k\) the same construction run at parameter \(j\) gives that level's \(Y_{2}\), and appending the free blocks \({\overline{v}}_{j + 1},\ldots,{\overline{v}}_{k}\) crosses it with \(\mathbb{A}^{3(k - j)}\) -- which is exactly \thmref{Theorem}{thm-componentcount}'s description of the inherited component indexed \(j\). That the vanishing ideal of the closure of a polynomial map's image is the kernel of the corresponding algebra homomorphism is the standard identification, over \(\overline{\mathbb{F}}\).

For birationality, \thmref{Lemma}{lem-y2-dim}'s injectivity argument (over \(U_{j - 1}\), via \thmref{Lemma}{lem-nondeg} at level \(j - 1\)) applies verbatim at parameter \(j\), and appending the free blocks \(q\) unchanged does not affect injectivity.

\end{proof}

\begin{remark}{a reducible variety is not itself ``rational'' -- component-by-component is what \thmref{Proposition}{prop-rational} answers}{auto4} \(Y_{1} \coloneqq  V\left( I_{k - 1} \right) \times \mathbb{A}^{3}\) (\thmref{Definition}{def-y1y2}) is \textbf{not} irreducible for \(k \geq 4\) -- it is a union of \(k - 2\) of the components above (\thmref{Theorem}{thm-componentcount}'s inductive step) -- so a single rational parametrization \textbf{of \(Y_{1}\) as one object} is not meaningful in the usual sense (rationality is a property of irreducible varieties). What \textbf{is} meaningful, and what \thmref{Proposition}{prop-rational} supplies, is a rational parametrization of each of \(Y_{1}\)'s irreducible pieces individually -- exactly the components inherited from \(V\left( I_{k - 1} \right)\).

\end{remark}

\begin{example}{the components, worked out at \(k = 3\)}{ex-components-k3} Uniformly by \(k\), since the components' \textbf{count} and \textbf{block pattern} are already known (\thmref{Theorem}{thm-componentcount}, \thmref{Definition}{def-primecomponents}) and only their explicit generators change:

\begin{itemize}
\item
  \textbf{At \(k = 2\)}: \(I_{2}\) is already prime (\thmref{Theorem}{thm-componentcount}'s base case), so \(V\left( I_{2} \right)\) has exactly \(k - 1 = 1\) component -- itself.
\item
  \textbf{At \(k = 3\)}: exactly \(k - 1 = 2\) components, both computed directly (Gröbner bases, not inferred from the pattern alone), matching \(I_{3} = P_{2} \cap P_{3}\):

  \begin{itemize}
  \item
    \(P_{2}\), the ``\(Y_{1}\)'' piece (block pattern \(\{ 1,2\}\), matching that the length-2 sub-bracket already vanishes with \({\overline{v}}_{3}\) free) -- generated by exactly \(I_{2}\)'s three generators from the first example, now inside \(R_{3}\) with \({\overline{v}}_{3}\) simply not appearing: \begin{equation}\fitdisplay{v_{1,2}v_{2,3} - v_{1,3}v_{2,2},\quad v_{1,3}v_{2,1} - v_{1,1}v_{2,3},\quad v_{1,1}v_{2,2} - v_{1,2}v_{2,1}.}\end{equation}
  \item
    \(P_{3}\), the ``\(Y_{2}\)'' piece (block pattern \(\{ 1,2,3\}\)) -- exactly the ideal already computed in the \(Y_{2}\)-defining-ideal example above: \begin{equation}\fitdisplay{v_{2,1}v_{3,1} + v_{2,2}v_{3,2} - v_{2,3}v_{3,3},\quad v_{1,1}v_{3,1} + v_{1,2}v_{3,2} - v_{1,3}v_{3,3},}\end{equation} \begin{equation}\fitdisplay{v_{1,2}v_{2,1}v_{3,2} - v_{1,1}v_{2,2}v_{3,2} - v_{1,3}v_{2,1}v_{3,3} + v_{1,1}v_{2,3}v_{3,3}.}\end{equation}
  \end{itemize}
\item
  \textbf{At \(k = 4\)}: exactly \(k - 1 = 3\) components, block patterns \(\{ 1,2\}\), \(\{ 1,2,3\}\), \(\{ 1,2,3,4\}\):

  \begin{itemize}
  \item
    The first two, \(P_{2}\) and \(P_{3}\), are inherited \textbf{verbatim} from \(k = 3\) above, with \({\overline{v}}_{4}\) simply left free -- not recomputed, per \thmref{Definition}{def-primecomponents}' own ``does not depend on which \(k\)'' clause.
  \item
    The third, \(P_{4}\), is the new ``\(Y_{2}\)'' stratum -- exactly the \(k = 4\) row of the earlier \(Y_{2}\)-defining-ideal example (generators omitted here to avoid repeating that display; same growth pattern, quartic with up to 8 terms per generator).
  \end{itemize}
\end{itemize}

\end{example}

\section{Cohen-Macaulayness}

\begin{proposition}{Eagon-Northcott / Hilbert-Burch, applied}{prop-detideal} Let \(M\) be a \(2 \times 3\) matrix with entries in a polynomial ring \(R\) over a field (arbitrary entries, not necessarily independent indeterminates), and \(J\) the ideal of its \(2 \times 2\) minors. If \(J\) has the \textbf{expected} codimension \(3 - 2 + 1 = 2\), then \(J\) is a perfect ideal of grade 2 -- \(\text{pd}(R/J) = 2\), hence \(R/J\) is Cohen-Macaulay.

\end{proposition}

\begin{proof}[{Proof of \thmref{Proposition}{prop-detideal}}] Bruns and Vetter \cite{BrunsVetter1988}, Theorem (2.7), p. 13, gives exactly this for an \textbf{arbitrary} matrix, as needed since one row of our \(2 \times 3\) matrix (\(S\left( {\overline{w}}_{k - 1} \right)\)) is a fixed polynomial vector, not a row of independent indeterminates. (The construction originates with Eagon and Northcott \cite{EagonNorthcott1962} it is the more general form given by Bruns and Vetter that is applied here.)

\end{proof}

\begin{remark}{why the generator count above wasn't proven this way}{auto5} One might expect \thmref{Proposition}{prop-detideal} to hand over the generator count as well, for free, via Bruns-Vetter's Theorem (16.36), p. 217, the Theorem of Hilbert-Burch. It does not. That theorem as actually stated gives \(I = a \cdot I_{m}(f)\) for a nonzerodivisor \(a\), and a generator count is not among its conclusions.

Nor could it be. At expected codimension the minors of an \textbf{arbitrary} \(2 \times 3\) matrix can be perfect and Cohen-Macaulay while one minor vanishes identically, or is a scalar combination of the other two. A single example settles it: over \(\mathbb{F}\lbrack x,y\rbrack\) the matrix \(\begin{pmatrix}
x & y & 0 \\
0 & 0 & 1
\end{pmatrix}\) has minors \(0,x,y\) and height 2, exactly as expected, yet \(J = (x,y)\) is perfect and Cohen-Macaulay with only \textbf{two} generators.

This is why \thmref{Theorem}{thm-mingens} is proven directly from linear independence rather than cited from this section's machinery.

\end{remark}

\begin{theorem}{\(R_{k}/I_{k}\) is Cohen-Macaulay}{thm-cm} For every \(k \geq 2\): \(R_{k}/I_{k}\) is Cohen-Macaulay.

\end{theorem}

\begin{proof}[{Proof of \thmref{Theorem}{thm-cm}}] By \thmref{Theorem}{thm-codim}, \(I_{k}\) has codimension 2, the expected value for the \(2 \times 3\) matrix of \thmref{Theorem}{thm-determinantal} presenting it. \thmref{Proposition}{prop-detideal} applies directly. (This uses \thmref{Theorem}{thm-codim} only -- not \thmref{Theorem}{thm-mingens}'s generator count, not \thmref{Theorem}{thm-componentcount}'s component count -- genuinely independent of both.)

\end{proof}

\section{Radicality}

Cohen-Macaulayness and the explicit component structure of \S{}5 combine, via Serre's criterion, to give radicality -- which then closes both the generator-count (\S{}4) / Cohen-Macaulay (\S{}6) story (making \(I_{k}\) itself, not just \(\sqrt{I_{k}}\), the object with those properties) and upgrades \S{}5's component count to a genuine primary decomposition.

\begin{lemma}{reduction via Serre}{lem-serre} Recall Serre's two conditions. They are written \(\left( R_{n} \right)\) and \(\left( S_{n} \right)\) throughout, always parenthesised, to keep them apart from this article's polynomial rings \(R_{k}\) --- the letters collide, the objects have nothing to do with each other, and \(R_{0}\) unparenthesised beside \(R_{k}\) invites exactly the wrong reading.

A Noetherian ring \(A\) satisfies \(\left( S_{n} \right)\) if \(\text{depth}\left( A_{P} \right) \geq \min(n,\dim A_{P})\) for every prime \(P \subset A\). It satisfies \(\left( R_{n} \right)\) if \(A_{P}\) is a \textbf{regular} local ring for every prime \(P\) of height \(\leq n\); in particular \(\left( R_{0} \right)\) asks for regularity at the minimal primes only, which for a reduced-or-not ring means that the local ring at each minimal prime is a field. Cohen-Macaulayness is the extremal case of the first: it asks for \(\text{depth}\left( A_{P} \right) = \dim A_{P}\) at \textbf{every} prime, which is \(n = \infty\) in effect.

A Cohen-Macaulay ring therefore satisfies \(\left( S_{n} \right)\) for every \(n\), since \(\text{depth}\left( A_{P} \right) = \dim A_{P} \geq \min(n,\dim A_{P})\) holds trivially whatever \(n\) is. This needs nothing beyond the two definitions just given, and it is standard for Serre's conditions in general: see Bruns and Vetter \cite{BrunsVetter1988}, or Eisenbud's \emph{Commutative Algebra with a View Toward Algebraic Geometry}. In particular \(R_{k}/I_{k}\) satisfies \(\left( S_{1} \right)\).

Serre's criterion for reducedness says that a Noetherian ring is reduced if and only if it satisfies both \(\left( R_{0} \right)\) and \(\left( S_{1} \right)\) \cite{Stacks031R}. Given \thmref{Theorem}{thm-cm}, then, radicality of \(I_{k}\) is equivalent to \(\left( R_{0} \right)\) alone: regularity of \(R_{k}/I_{k}\) at each of its minimal primes.

\end{lemma}

The four lemmas that follow are the standard Jacobian criterion, split into its two genuinely separate ingredients -- one linear-algebraic (\thmref{Lemma}{lem-cotangent}), one dimension-theoretic (\thmref{Lemma}{lem-localdim}) -- before either is used. The order is deliberate: the rank bound and the regularity criterion each draw on \textbf{both} ingredients and on neither each other, which a single pair of mutually-citing lemmas (as in an earlier draft of this section) obscures.

\begin{lemma}{the cotangent identity}{lem-cotangent} Let \(S = \mathbb{F}\left\lbrack x_{1},\ldots,x_{n} \right\rbrack\), \(I = \left( f_{1},\ldots,f_{m} \right)\), \(A = S/I\), and \(x \in V(I)\) an \(\mathbb{F}\)-rational point, with \(\mathfrak{n} \subset S\) the maximal ideal at \(x\) and \(\mathfrak{m} = \mathfrak{n}/I \subset A\) its image. Then \begin{equation}\fitdisplay{\dim_{\mathbb{F}}\mathfrak{m}/\mathfrak{m}^{2} = n - \text{ rank }\text{ Jac}(x).}\end{equation}

\end{lemma}

\begin{proof}[{Proof of \thmref{Lemma}{lem-cotangent}}] Only linear algebra is needed here; in particular, no hypothesis on \(A\) is required. Sending \(g\) to its gradient at \(x\) identifies the cotangent space \(\mathfrak{n}/\mathfrak{n}^{2}\) with \(\mathbb{F}^{n}\). Under that identification the image of \(I\) is the span of the linear parts at \(x\) of \(f_{1},\ldots,f_{m}\): a general element of \(I\) is \(\sum_{i}g_{i}f_{i}\), whose linear part at \(x\) is \(\sum_{i}g_{i}(x)\) times that of \(f_{i}\), since every \(f_{i}\) vanishes at \(x\). That span is the row space of \(\text{Jac}(x)\), of dimension \(\text{rank }\text{ Jac}(x)\). As \(\mathfrak{m}/\mathfrak{m}^{2} \cong \left( \mathfrak{n}/\mathfrak{n}^{2} \right)/\text{ im}(I)\), the claim follows.

\end{proof}

\begin{lemma}{the local dimension at a rational point}{lem-localdim} Let \(S = \mathbb{F}\left\lbrack x_{1},\ldots,x_{n} \right\rbrack\), \(I = \left( f_{1},\ldots,f_{m} \right)\), \(A = S/I\) equidimensional of dimension \(n - c\), and \(x \in V(I)\) an \(\mathbb{F}\)-rational point with maximal ideal \(\mathfrak{m} \subset A\). Write \(\text{Jac}(x) \in \mathbb{F}^{m \times n}\) for the Jacobian \(\left( \partial f_{i}/\partial x_{j} \right)\) \textbf{evaluated at \(x\)}: a matrix of scalars, not the matrix of polynomials it is evaluated from. (The polynomial Jacobian itself appears later as \(J_{k}\) in \thmref{Lemma}{lem-jacrecursion}; the two are not the same object, and only the evaluated one is used in this section.) Then \(\dim A_{\mathfrak{m}} = n - c\).

\end{lemma}

\begin{proof}[{Proof of \thmref{Lemma}{lem-localdim}}] The local dimension is computed componentwise: \begin{equation}\fitdisplay{\dim A_{\mathfrak{m}} = \max\left\{ \dim(A/P)_{\mathfrak{m}}:P \in \text{ Min}(A),P \subseteq \mathfrak{m} \right\},}\end{equation} and each \(A/P = S/\widetilde{P}\) is an affine domain over a field, for which \(\dim\left( S/\widetilde{P} \right)_{\mathfrak{M}} = \dim S/\widetilde{P}\) at \textbf{every} maximal ideal \(\mathfrak{M}\) -- affine domains over a field are catenary and equicodimensional (Eisenbud, \emph{Commutative Algebra}, Cor. 13.4; equivalently the Stacks Project, tag 00OS \cite{Stacks00OS} -- a dimension formula for finitely generated field extensions, holding over an arbitrary field, not needing a perfect-field hypothesis by accident). By equidimensionality every such \(\dim S/\widetilde{P}\) equals \(n - c\).

\end{proof}

\begin{lemma}{Jacobian criterion: the rank bound}{lem-jacobian-rankbound} The hypotheses are those of \thmref{Lemma}{lem-localdim}. Then \(\text{rank }\text{ Jac}(x) \leq c\).

\end{lemma}

\begin{proof}[{Proof of \thmref{Lemma}{lem-jacobian-rankbound}}] In any Noetherian local ring one has \(\dim_{\mathbb{F}}\mathfrak{m}/\mathfrak{m}^{2} \geq \dim A_{\mathfrak{m}}\). By \thmref{Lemma}{lem-cotangent} the left-hand side is \(n - \text{ rank }\text{ Jac}(x)\), and by \thmref{Lemma}{lem-localdim} the right-hand side is \(n - c\); rearranging gives \(\text{rank }\text{ Jac}(x) \leq c\).

\end{proof}

\begin{lemma}{Jacobian criterion: regularity at rank \(= c\)}{lem-jacobian-regularity} The hypotheses are those of \thmref{Lemma}{lem-localdim}. If \(\text{rank }\text{ Jac}(x) = c\), then \(A_{\mathfrak{m}}\) is regular of dimension \(n - c\); in particular \(x\) lies on a \textbf{unique} component of \(V(I)\), whose local ring there is a field.

\end{lemma}

\begin{proof}[{Proof of \thmref{Lemma}{lem-jacobian-regularity}}] If \(\text{rank }\text{ Jac}(x) = c\) then \(\dim_{\mathbb{F}}\mathfrak{m}/\mathfrak{m}^{2} = n - c\) by \thmref{Lemma}{lem-cotangent}, while \(\dim A_{\mathfrak{m}} = n - c\) by \thmref{Lemma}{lem-localdim}; equality of the two is the definition of regularity. A regular local ring is a domain, so \(\mathfrak{m}\) contains a unique minimal prime of \(A\), and the local ring at that prime is a further localization of a regular local ring at its own maximal ideal in dimension 0, hence a field.

\end{proof}

\begin{lemma}{the Jacobian's own recursion}{lem-jacrecursion} Write \(J_{k}\) for the \(3 \times 3k\) Jacobian of \({\overline{w}}_{k}\), \(\lbrack u\rbrack_{\times}\) for the skew matrix with \(\lbrack u\rbrack_{\times}v = u \times v\). Then \(J_{2} = \left\lbrack - \left\lbrack {\overline{v}}_{2} \right\rbrack_{\times}|\left\lbrack {\overline{v}}_{1} \right\rbrack_{\times} \right\rbrack\), and for \(k \geq 3\), \begin{equation}\fitdisplay{J_{k} = \left\lbrack - \left\lbrack {\overline{v}}_{k} \right\rbrack_{\times}SJ_{k - 1}|\left\lbrack {S\left( {\overline{w}}_{k - 1} \right)} \right\rbrack_{\times} \right\rbrack.}\end{equation}

\end{lemma}

\begin{proof}[{Proof of \thmref{Lemma}{lem-jacrecursion}}] For the second block, \({\overline{w}}_{k} = {S\left( {\overline{w}}_{k - 1} \right)} \times {\overline{v}}_{k} = \left\lbrack {S\left( {\overline{w}}_{k - 1} \right)} \right\rbrack_{\times}{\overline{v}}_{k}\) is linear in \({\overline{v}}_{k}\) with coefficient matrix free of \({\overline{v}}_{k}\). For the first, differentiating \({\overline{w}}_{k} = {S\left( {\overline{w}}_{k - 1} \right)} \times {\overline{v}}_{k}\) with respect to any variable \(t\) among \({\overline{v}}_{1},\ldots,{\overline{v}}_{k - 1}\) gives \(\partial_{t}{\overline{w}}_{k} = \left( S\partial_{t}{\overline{w}}_{k - 1} \right) \times {\overline{v}}_{k} = - \left\lbrack {\overline{v}}_{k} \right\rbrack_{\times}S\partial_{t}{\overline{w}}_{k - 1}\), using \(a \times b = - \lbrack b\rbrack_{\times}a\); assembling columns gives the stated product.

\end{proof}

\begin{example}{\(J_{3}\), explicitly}{auto6} With variables ordered \(v_{1,1},v_{1,2},v_{1,3},v_{2,1},v_{2,2},v_{2,3},v_{3,1},v_{3,2},v_{3,3}\), the \(3 \times 9\) Jacobian of \({\overline{w}}_{3}\) (the three generators of \(I_{3}\) from the earlier example), column-blocked to match \thmref{Lemma}{lem-jacrecursion}'s formula \(J_{3} = \left\lbrack - \left\lbrack {\overline{v}}_{3} \right\rbrack_{\times}SJ_{2}|\left\lbrack {S\left( {\overline{w}}_{2} \right)} \right\rbrack_{\times} \right\rbrack\) -- shown here as the two blocks separately, since side by side the full \(3 \times 9\) matrix does not fit the page width. First, \(- \left\lbrack {\overline{v}}_{3} \right\rbrack_{\times}SJ_{2}\) (\(3 \times 6\)): \begin{equation}\fitdisplay{\begin{pmatrix}
 - v_{2,2}v_{3,2} + v_{2,3}v_{3,3} & v_{2,1}v_{3,2} & - v_{2,1}v_{3,3} & v_{1,2}v_{3,2} - v_{1,3}v_{3,3} & - v_{1,1}v_{3,2} & v_{1,1}v_{3,3} \\
v_{2,2}v_{3,1} & - v_{2,1}v_{3,1} + v_{2,3}v_{3,3} & - v_{2,2}v_{3,3} & - v_{1,2}v_{3,1} & v_{1,1}v_{3,1} - v_{1,3}v_{3,3} & v_{1,2}v_{3,3} \\
 - v_{2,3}v_{3,1} & - v_{2,3}v_{3,2} & v_{2,1}v_{3,1} + v_{2,2}v_{3,2} & v_{1,3}v_{3,1} & v_{1,3}v_{3,2} & - v_{1,1}v_{3,1} - v_{1,2}v_{3,2}
\end{pmatrix}}\end{equation} and then \(\left\lbrack {S\left( {\overline{w}}_{2} \right)} \right\rbrack_{\times}\) (\(3 \times 3\), the skew matrix of \(S\left( {\overline{w}}_{2} \right)\)'s three coordinates from the earlier matrix example): \begin{equation}\fitdisplay{\begin{pmatrix}
0 & v_{1,2}v_{2,1} - v_{1,1}v_{2,2} & - v_{1,3}v_{2,1} + v_{1,1}v_{2,3} \\
 - v_{1,2}v_{2,1} + v_{1,1}v_{2,2} & 0 & - v_{1,3}v_{2,2} + v_{1,2}v_{2,3} \\
v_{1,3}v_{2,1} - v_{1,1}v_{2,3} & v_{1,3}v_{2,2} - v_{1,2}v_{2,3} & 0
\end{pmatrix}}\end{equation} Both blocks computed directly, not merely asserted from the recursive formula.

\end{example}

\begin{remark}{a false shortcut, caught and corrected}{auto7} A naive reading of \thmref{Lemma}{lem-jacrecursion} suggests appending \({\overline{v}}_{k} = 0\) preserves rank on an inherited component -- this is false: the old-variable block is left-multiplied by \(- \left\lbrack {\overline{v}}_{k} \right\rbrack_{\times}S\), not merely padded with zero rows. At the first inherited case (\(k = 2 \rightarrow 3\)), appending \({\overline{v}}_{3} = 0\) actually drops the rank to 0. The correct witness (\thmref{Lemma}{lem-witness-inherited} below) instead picks \({\overline{v}}_{k} = e_{i}\) for a standard basis vector \textbf{not} lying in \(S \cdot \text{ im }J_{k - 1}(q)\) -- a genuine choice, not a formality.

\end{remark}

\begin{lemma}{rank-2 witness on the new component}{lem-witness-new} For every \(k \geq 3\), there is a point \(x = (p,0)\), \(p\) the \thmref{Lemma}{lem-nondeg} non-degeneracy witness at level \(k - 1\), with \(x \in V\left( I_{k} \right)\) and \(\text{rank }J_{k(x)} = 2\). At \(k = 2\) the same conclusion holds with \(x = (p,0)\) for any \(p \neq 0\).

\end{lemma}

\begin{proof}[{Proof of \thmref{Lemma}{lem-witness-new}}] \textbf{The case \(k = 2\) separately}, because the statement above cannot be read at \(k = 2\): it would refer to a non-degeneracy witness ``at level 1'', and neither \({\overline{w}}_{1}\) nor \thmref{Lemma}{lem-nondeg} exists there --- \thmref{Definition}{def-wk} starts the recursion at \(j = 2\). Nor does \thmref{Lemma}{lem-jacrecursion}'s \(k \geq 3\) formula apply. Argue directly instead: \(J_{2} = \left\lbrack - \left\lbrack {\overline{v}}_{2} \right\rbrack_{\times}|\left\lbrack {\overline{v}}_{1} \right\rbrack_{\times} \right\rbrack\) carries no \(S\) at all, so for any \(p \neq 0\), \(J_{2}(p,0) = \left\lbrack 0|\lbrack p\rbrack_{\times} \right\rbrack\), which has rank 2 by the kernel computation below. And \((p,0) \in V\left( I_{2} \right)\) since \({\overline{w}}_{2}(p,0) = p \times 0 = 0\).

\textbf{For \(k \geq 3.\)} By \thmref{Lemma}{lem-jacrecursion}, \(J_{k(p,0)} = \left\lbrack 0|\left\lbrack {S\left( {\overline{w}}_{k - 1}(p) \right)} \right\rbrack_{\times} \right\rbrack\) since \(- \lbrack 0\rbrack_{\times}SJ_{k - 1}(p) = 0\); and \({S\left( {\overline{w}}_{k - 1}(p) \right)} \neq 0\) since \({\overline{w}}_{k - 1}(p) \neq 0\) (\thmref{Lemma}{lem-nondeg}) and \(S\) is invertible. For \(u \neq 0\), \(\ker\lbrack u\rbrack_{\times} = \left\{ v:u \times v = 0 \right\} = \mathbb{F}u\) (a cross product vanishes exactly on linearly dependent pairs), so \(\left\lbrack {S\left( {\overline{w}}_{k - 1}(p) \right)} \right\rbrack_{\times}\) has rank 2, giving \(\text{rank }J_{k(p,0)} = 2\).

\end{proof}

\begin{lemma}{rank-2 witness on an inherited component}{lem-witness-inherited} For every \(k \geq 3\) and every component \(Z\) of \(V\left( I_{k - 1} \right)\) with a rank-2 witness \(q \in Z\) (\thmref{Lemma}{lem-witness-new}, or this lemma applied at \(k - 1\)), there is \(x = \left( q,e_{i} \right) \in Z \times \mathbb{A}^{3} \subset V\left( I_{k} \right)\) with \(\text{rank }J_{k(x)} = 2\).

\end{lemma}

\begin{proof}[{Proof of \thmref{Lemma}{lem-witness-inherited}}] Let \(W \coloneqq  S \cdot \text{ im }J_{k - 1}(q)\), a 2-dimensional subspace of \(\mathbb{F}^{3}\) since \(S\) is invertible and \(\text{rank }J_{k - 1}(q) = 2\). Since \(e_{1},e_{2},e_{3}\) span \(\mathbb{F}^{3}\) and \(W \neq \mathbb{F}^{3}\), some \(e_{i} \notin W\); fix such an \(i\). Since \(q \in V\left( I_{k - 1} \right)\), \({\overline{w}}_{k - 1}(q) = 0\), so \thmref{Lemma}{lem-jacrecursion} gives \(J_{k\left( q,e_{i} \right)} = \left\lbrack - \left\lbrack e_{i} \right\rbrack_{\times}SJ_{k - 1}(q)|0 \right\rbrack\), of rank \(\dim(\left\lbrack e_{i} \right\rbrack_{\times}W)\) since \(SJ_{k - 1}(q)\) has image exactly \(W\). As \(\ker\left\lbrack e_{i} \right\rbrack_{\times} = \mathbb{F}e_{i}\) and \(W \cap \mathbb{F}e_{i} = 0\) (else \(e_{i} \in W\)), \(\left\lbrack e_{i} \right\rbrack_{\times}\) is injective on \(W\), so \(\dim(\left\lbrack e_{i} \right\rbrack_{\times}W) = 2\).

\end{proof}

\begin{theorem}{\(I_{k}\) is radical}{thm-radical} For every \(k \geq 2\): \(I_{k}\) is radical.

\end{theorem}

\begin{proof}[{Proof of \thmref{Theorem}{thm-radical}}] By induction on \(k\) (mirroring \thmref{Theorem}{thm-componentcount}'s own induction), \thmref{Lemma}{lem-witness-new} and \thmref{Lemma}{lem-witness-inherited} exhibit a rank-2-Jacobian point on \textbf{every} one of the \(k - 1\) components of \thmref{Theorem}{thm-componentcount}. \thmref{Lemma}{lem-jacobian-regularity} then gives \(\left( R_{0} \right)\) at every one (and, by its ``unique component'' clause, confirms each witness lies only on its intended component -- an independent check on \thmref{Theorem}{thm-componentcount}'s count, not merely assumed); \thmref{Lemma}{lem-serre} gives \(\left( S_{1} \right)\), using \thmref{Theorem}{thm-cm}'s Cohen-Macaulayness. Serre's criterion \cite{Stacks031R} gives radicality.

\end{proof}

This has been verified independently in Sage, for the signed encoding, at every \(k = 2,\ldots,6\): the Jacobian recursion (\thmref{Lemma}{lem-jacrecursion}) checked as a polynomial identity, not just at points, and every witness confirmed to lie on \(V\left( I_{k} \right)\) with \(\text{rank }J_{k} = 2\) (anc/code/referee\_radicality\_S\_check.sage). The sharper claim --- that each witness lies on \textbf{exactly} its intended component --- is confirmed by an independently-computed primary decomposition at every \(k = 3,4,5,6\), which returns \(k - 1\) components in each case, in agreement with \thmref{Theorem}{thm-componentcount}; the \(k = 6\) decomposition takes some hours.

\section{The Primary Decomposition Theorem}

Every ingredient is now in hand for the fully general statement -- component count (\S{}5) and radicality (\S{}7) together, neither sufficient alone. One convention has to be made explicit first, because it is the hinge on which the statement's claimed generality turns.

\begin{remark}{the geometric arguments run over \(\overline{\mathbb{F}}\), and the conclusion descends}{rem-base-field} Sections 5-7 argue geometrically: they speak of \(V\left( I_{k} \right)\), of its irreducible components, and of the closure of the image of a parametrization, and at the decisive step they identify \(\text{Min}\left( I_{k} \right)\) with the vanishing ideals of those components. That identification is a Nullstellensatz statement and is \textbf{false over a general field} -- over \(\mathbb{F}_{3}\), for instance, \(V\left( I_{k} \right)\left( \mathbb{F}_{3} \right)\) is a finite set of points and carries no such information at all. So throughout \S{}\S{}5-7, and in the theorem below, \(V( \cdot )\) means the vanishing locus in \(\mathbb{A}^{3k}\left( \overline{\mathbb{F}} \right)\) over a fixed algebraic closure, and ``component'' means an irreducible component there. This is the usual convention and is what the classical literature on the commuting variety assumes as well (Artin-Hochster is stated over an algebraically closed field in \cite{MajidiZolbaninSnapp}).

The conclusion nonetheless descends to \(\mathbb{F}\) itself, so the theorem may be read over any field of characteristic \(\neq 2\), as stated. The reason is that each \(P_{j}\) is not merely prime but \textbf{absolutely} prime. By \thmref{Definition}{def-primecomponents}, \(P_{j}\) \textbf{is} the kernel of an \(\mathbb{F}\)-algebra homomorphism \(\varphi_{j}:R_{k} \rightarrow \mathbb{F}\left\lbrack t_{1},\ldots,t_{3k - 2} \right\rbrack\) into a polynomial ring, so \(R_{k}/P_{j}\) embeds in that polynomial ring. Since \(\overline{\mathbb{F}}\) is flat over \(\mathbb{F}\), tensoring preserves the embedding: \begin{equation}\fitdisplay{\overline{\mathbb{F}} \otimes_{\mathbb{F}}\left( R_{k}/P_{j} \right) \hookrightarrow \overline{\mathbb{F}}\left\lbrack t_{1},\ldots,t_{m} \right\rbrack,}\end{equation} whose target is a domain, so the extension of \(P_{j}\) to \(\overline{\mathbb{F}} \otimes_{\mathbb{F}}R_{k}\) is again prime. Formation of a finite intersection of ideals likewise commutes with the flat base change \(\mathbb{F} \rightarrow \overline{\mathbb{F}}\). Hence the displayed decomposition, established over \(\overline{\mathbb{F}}\), holds already over \(\mathbb{F}\), with the same \(k - 1\) primes.

\end{remark}

\begin{theorem}{Primary Decomposition Theorem}{thm-primarydecomp} For every \(k \geq 2\), over any field of characteristic \(\neq 2\): \(I_{k}\) has exactly \(k - 1\) irreducible, radical, primary components, namely \(P_{2},\ldots,P_{k}\) of \thmref{Definition}{def-primecomponents}, and \begin{equation}\fitdisplay{I_{k} = P_{2} \cap \cdots \cap P_{k}}\end{equation} is an irredundant primary decomposition into \(k - 1\) primes.

\end{theorem}

\begin{proof}[{Proof of \thmref{Theorem}{thm-primarydecomp}}] The two halves were proven independently above, so this combines them in full rather than asserting the conjunction in one line.

\textbf{The decomposition itself.} By \thmref{Theorem}{thm-radical}, \(I_{k}\) is radical, so \(I_{k} = \sqrt{I_{k}}\). A radical ideal in a Noetherian ring equals the intersection of its own minimal primes: \begin{equation}\fitdisplay{I_{k} = \cap_{P \in \text{ Min}\left( I_{k} \right)}P.}\end{equation} This holds because \(\sqrt{I_{k}}\) is the intersection of \textbf{all} primes containing \(I_{k}\) -- Krull's theorem, not the definition of the radical, which is instead the set of elements with a power in \(I_{k}\) -- and the minimal primes already suffice. Every non-minimal prime containing \(I_{k}\) contains a minimal one, and is therefore redundant in the intersection.

\thmref{Theorem}{thm-componentcount} identifies \(\text{Min}\left( I_{k} \right)\) explicitly, over \(\overline{\mathbb{F}}\) as \thmref{Remark}{rem-base-field} requires: exactly \(k - 1\) primes, one per block pattern \(j = 2,\ldots,k\), named \(P_{2},\ldots,P_{k}\) in \thmref{Definition}{def-primecomponents}. Substituting them gives the displayed equation.

\textbf{Why the components are primary.} Each \(P_{j}\) is prime, so it is trivially its own primary-decomposition component, being \(P_{j}\)-primary with associated prime \(P_{j}\). Nothing beyond primality has to be argued.

\textbf{Why they are prime.} Nothing is left to prove here either. \thmref{Definition}{def-primecomponents} \textbf{defines} \(P_{j}\) as \(\ker\varphi_{j}\), and the kernel of a homomorphism into a domain is prime. That argument covers \(j = 2\) along with the rest, at no extra cost. \thmref{Lemma}{lem-y2-irred} would give the same conclusion for \(j \geq 3\), but it is stated for \(k \geq 3\) and so misses \(j = 2\); \(P_{2}\) is \(I_{2}\) itself, the ideal of maximal minors of a generic \(2 \times 3\) matrix, which is classically prime. Neither of those routes is needed.

Extending \(P_{j}\) from \(R_{j}\) to \(R_{k}\) by adjoining the free blocks \({\overline{v}}_{j + 1},\ldots,{\overline{v}}_{k}\) preserves primality, since \(R_{k}/P_{j}R_{k}\) is a polynomial ring over the domain \(R_{j}/P_{j}\).

\textbf{Irredundancy.} No \(P_{j}\) can be omitted, i.e. no \(P_{j}\) contains \(\cap_{i \neq j}P_{i}\). This is exactly \thmref{Theorem}{thm-componentcount}'s own no-containment conclusion, obtained from \thmref{Proposition}{prop-y2-not-in-y1} and \thmref{Proposition}{prop-y1-not-in-y2} applied at every level of that theorem's induction. It needs no further proof here.

\end{proof}

\begin{remark}{summary}{auto8} \thmref{Theorem}{thm-codim} (codimension 2), \thmref{Theorem}{thm-mingens} (3 minimal generators), \thmref{Theorem}{thm-cm} (Cohen-Macaulay), and \thmref{Theorem}{thm-radical} (radical) together say exactly what a single bundled ``\(I_{k}\) has codimension 2, is radical, has exactly 3 minimal generators, and is Cohen-Macaulay'' statement would -- now as four independently-proven results, each established at the point in this article where its actual dependencies (and no others) were available, rather than one claim proven in parts across several sections.

\end{remark}

\section{Summary}

For every \(k \geq 2\), over any field \(\mathbb{F}\) of characteristic \(\neq 2\):

\begin{itemize}
\item
  \(I_{k}\) has codimension 2 (\thmref{Theorem}{thm-codim}), presented as the ideal of \(2 \times 2\) minors of an explicit \(2 \times 3\) matrix (\thmref{Theorem}{thm-determinantal}) -- a genuine determinantal ideal, not merely determinantal-looking.
\item
  \(I_{k}\) has exactly 3 minimal generators (\thmref{Theorem}{thm-mingens}), proven directly from linear independence, independent of the determinantal presentation.
\item
  \(R_{k}/I_{k}\) is Cohen-Macaulay (\thmref{Theorem}{thm-cm}), from the determinantal presentation and classical Eagon-Northcott/Hilbert-Burch theory (\thmref{Proposition}{prop-detideal}).
\item
  \(I_{k}\) is radical (\thmref{Theorem}{thm-radical}), via Serre's criterion: Cohen-Macaulayness gives \(\left( S_{1} \right)\) (\thmref{Lemma}{lem-serre}), and an explicit rank-2 Jacobian witness on every component gives \(\left( R_{0} \right)\) (\thmref{Lemma}{lem-witness-new}, \thmref{Lemma}{lem-witness-inherited}).
\item
  \(V\left( I_{k} \right)\) has exactly \(k - 1\) irreducible components (\thmref{Theorem}{thm-componentcount}), each rational -- in fact birational to an explicit affine parameter space (\thmref{Proposition}{prop-rational}). Radicality and this component count together -- rationality plays no part in the assembly -- give the irredundant primary decomposition of the Primary Decomposition Theorem (\thmref{Theorem}{thm-primarydecomp}): \(I_{k} = P_{2} \cap \cdots \cap P_{k}\).
\end{itemize}

\section{Methods}\label{sec-methods}

This article was produced by a workflow that is not the usual one. Since the workflow bears on how some of the claims above should be read, it is set out here rather than left to be inferred.

\textbf{Computer algebra, in two roles that should not be confused.} SageMath and Macaulay2 were used throughout, for two quite different purposes. The first is \textbf{conjecture}: the component count of \thmref{Theorem}{thm-componentcount}, its block-involvement pattern, and the shape of the primary decomposition were found by computing primary decompositions of \(I_{k}\) at small \(k\) and reading off the pattern, well before any of it was proved. The second is \textbf{verification}: once proved, the general statements were checked against independent computation --- the determinantal structure of \thmref{Theorem}{thm-determinantal} by Gröbner-basis equality in both directions and, separately, against ordinary \(2 \times 2\) matrix multiplication, the two routes sharing no code; the Jacobian recursion of \thmref{Lemma}{lem-jacrecursion} as a polynomial identity rather than at sampled points; the witness points of \thmref{Lemma}{lem-witness-new} and \thmref{Lemma}{lem-witness-inherited} confirmed to lie on \(V\left( I_{k} \right)\) with the stated Jacobian rank, and on exactly the intended component; and the decomposition itself recomputed at \(k = 3,4,5,6\).

Because rank, dimension and radicality are \textbf{not} characteristic-independent --- a matrix of rank 3 over \(\mathbb{Q}\) may drop rank modulo \(p\) --- computations over \(\mathbb{Q}\) license nothing about \(\mathbb{F}_{p}\), and the four main results were therefore re-run over \(\mathbb{F}_{2}\), \(\mathbb{F}_{3}\), \(\mathbb{F}_{4}\), \(\mathbb{F}_{5}\), \(\mathbb{F}_{7}\) and \(\mathbb{F}_{9}\) as well. No deviation appeared. The scripts are named at the point of use, and are included with this submission as ancillary files under \texttt{anc/code/}, with their full output transcripts under \texttt{anc/runs/}; the paths cited in the text are the paths inside that directory. The bibliography is written for a reader and carries no access notes; \texttt{anc/runs/REFERENCE-CERTIFICATE.md} records instead, for each source, the exact statement relied on, whether it was read in the source or only confirmed from publisher metadata, and the two citation errors that earlier drafts carried. The explicit computations they produce are displayed in \cref{sec-appendix}.

Nothing in this article is conjectural, or checked only at small \(k\): every statement is proven for \textbf{every} \(k \geq 2\), over an \textbf{arbitrary} field of the stated characteristic. The computational thread corroborates the general proofs; nowhere does it substitute for one.

\textbf{Source of truth, and a partial formalization.} The manuscript is written in Typst, and that source is canonical: the LaTeX is generated from it mechanically, so the two cannot drift apart. The same source supplied the statements for a partial formalization in Lean 4 against \emph{mathlib}. Machine-checked, with no proof gaps and depending on no axioms beyond the three that every \emph{mathlib} development uses: the basic identity (\thmref{Theorem}{thm-basic-identity}), in the strengthened form that drops the traceless hypothesis; the minimal-generator theorem (\thmref{Theorem}{thm-mingens}) for the actual \({\overline{w}}_{k}\), together with the whole chain beneath it --- the general lemma that independent same-degree generators are minimal (\thmref{Lemma}{lem-mingens-generic}), the homogeneity of \({\overline{w}}_{k}\), the freshness of each new variable block, and the linear independence itself (\thmref{Lemma}{lem-mingens-indep}); the Jacobian recursion (\thmref{Lemma}{lem-jacrecursion}); and the smaller computational claims of \S{}3 and \S{}5, including the identity of \thmref{Remark}{rem-why-D}.

\textbf{Not} formalized: everything resting on determinantal-ideal theory or on dimension theory --- the codimension (\thmref{Theorem}{thm-codim}), Cohen-Macaulayness (\thmref{Theorem}{thm-cm}), radicality (\thmref{Theorem}{thm-radical}), the component count (\thmref{Theorem}{thm-componentcount}) and the Primary Decomposition Theorem (\thmref{Theorem}{thm-primarydecomp}) --- for which the necessary libraries do not presently exist. The formalization is therefore a \textbf{partial} guarantee, and is described as one: it should not be read as covering the article's main results.

\textbf{On the explicitness of the exposition.} This article is more explicit than a paper of its length usually is. Steps that would ordinarily be left to the reader are written out, definitions that would ordinarily be recalled in passing are stated, and there are more worked examples than the proofs strictly require --- the \(2 \times 3\) matrix of \thmref{Theorem}{thm-determinantal} displayed at three successive \(k\), the components at \(k = 3\) and \(k = 4\) given by their generators, the non-degeneracy witness tabulated, and the explicit computations collected in \cref{sec-appendix}.

That was deliberate, for two reasons. The first concerns the reader: a manuscript produced with substantial machine assistance, and only partially formalized, invites a scepticism that is entirely reasonable, and the honest answer to it is not reassurance but exposure --- to write the argument out at a level of detail at which a determined reader can check it without reconstructing the missing steps, and to say plainly, as \cref{sec-methods} does, which parts are machine-checked and which are not. The second concerns the author: writing at this level of explicitness is how the author satisfied himself that he understood the argument rather than merely followed it, and several of the corrections recorded below were found precisely because a step that had seemed obvious stopped being obvious once it had to be written down in full.

The same reasoning produced a large body of short standalone notes kept alongside the research, one question to a note. Their first purpose was the author's own understanding. Their second turned out to be discovery: the signed encoding \(\text{enc}\) used throughout this article --- which replaced an earlier coordinate choice that forced a coordinate swap and two factors of 2 --- emerged from one such note, written to answer nothing more ambitious than whether a different encoding might be tidier.

\textbf{Adversarial review.} The manuscript was put through five rounds of deliberately adversarial review, each instructed to find errors rather than to approve, and each required to verify by computation rather than by reading. This was not a formality. The rounds found: a genuine gap in the proof of \thmref{Lemma}{lem-mingens-indep}, where non-degeneracy was doing work that only linear independence can do; explicit displays that had not been recomputed after a change of encoding, and so still showed the superseded values; a false statement about the state of the literature, contradicted by the very sources cited for it; a citation to the wrong tag of a reference work, made under an explicit assurance that it had been verified; a stretch of prose that described the formulas printed beside it incorrectly; and the gap in generality that \thmref{Remark}{rem-base-field} now closes. All are corrected above. The mention is not decorative: a reader is entitled to know what kind of error an adversarial pass actually catches, and to calibrate accordingly.

\textbf{What this workflow is, and is not, good for.} Its strength is coverage of the mechanical: a recomputation notices a stale display, an inconsistent index or an unverified citation reliably, and does so at a scale and patience that re-reading does not sustain. Its weakness is that none of that establishes that the argument is \textbf{right}; every substantive correction above was found by an adversarial reading, not by a passing test. The machine checks and the Lean formalization narrow the space in which an error can hide; they do not empty it, and the responsibility statement below is meant literally.

\section{Acknowledgements}

The author first learned of the ring of commuting matrices some thirty years ago from Professor Jan-Erik Roos, in a course he held on the then fledgling field of computer algebra. It is a pleasant debt to record here, given how much of what follows was found with a computer algebra system open beside the page.

The computations were carried out in SageMath and Macaulay2, and the formalization in Lean 4 with \emph{mathlib}; this article would have a different and smaller content without them. The nature and extent of the assistance from a large language model is set out in \S{}10 and in the disclosure below.

\section{Disclosure of AI assistance}

This manuscript was produced with substantial assistance from a large language model, Claude (Anthropic), and the extent of it is stated here rather than left to be inferred. The itemised account is \S{}10: which results were machine-checked and which were not, what the computer algebra established and what it merely corroborated, and what the adversarial review found. Claude is not a co-author, and is not credited as one.

Every computational claim has been checked along a route independent of the statement being tested, and the reference list has been checked against the sources cited --- with the one exception that where a source is a printed book outside this project's open-access policy, an independently verified equivalent is cited alongside it. The author has verified the results to the best of his ability, made the final edits, and assumes full responsibility for any errors, mistakes or gaps that remain.

\section{Appendix: the explicit computations}\label{sec-appendix}

The results above are stated with their computations named rather than displayed, so that the argument stays legible. This appendix displays them. Everything here is \textbf{corroboration}: no proof in the article rests on a computation, and nothing below is offered as one. Each block names the script that produced it, and the scripts ship with this submission as ancillary files under \texttt{anc/code/}, with their full transcripts under \texttt{anc/runs/}.

One naming point, since the scripts ship as they are. Most of them were written while the sign twist was called \(D\), and they still write \texttt{D} internally for the matrix this article calls \(S\). It is the same matrix, \(\text{diag}( - 1, - 1,1)\); only the article's name for it changed. Their contents were not rewritten for the rename, on the grounds that editing working verification code for cosmetic reasons is a worse risk than a sentence saying so.

Every polynomial below was recomputed for this appendix from the recursion of \thmref{Definition}{def-wk} rather than transcribed from an earlier draft. That distinction is not pedantry here: the article changed encoding partway through its life, from the reshuffle \(\tau(a,b,c) = (b,a,2c)\) to \(S\), and a computed polynomial cannot be carried across such a change by substituting symbols --- it has to be recomputed. Output that still carries a literal coefficient \(2\) in the \(\overline{v}\)-coordinates is the tell that it was not.

\subsection{Both sides of the basic identity}\label{app-identity}

\thmref{Theorem}{thm-basic-identity} asserts \(\text{enc}\left( \left\lbrack A_{0},B_{0} \right\rbrack \right) = {S\left( \text{enc}\left( A_{0} \right) \times \text{ enc}\left( B_{0} \right) \right)}\). Expanded over generic entries \(a_{11},\ldots,b_{22}\), the two sides are, component by component:

\begin{verbatim}
lhs_1 = rhs_1 = -2*a21*b12 + 2*a12*b21
lhs_2 = rhs_2 = -a12*b11 + a21*b11 + a11*b12 - a22*b12 - a11*b21 + a22*b21
                + a12*b22 - a21*b22
lhs_3 = rhs_3 = -a12*b11 - a21*b11 + a11*b12 - a22*b12 + a11*b21 - a22*b21
                + a12*b22 + a21*b22
\end{verbatim}

The difference is \((0,0,0)\) identically. Chaining once more gives the triple commutator \(\text{enc}\left( \left\lbrack \lbrack A,B\rbrack,C \right\rbrack \right)\), which the recursion computes as \(S\left( {S\left( \text{enc}(A) \times \text{ enc}(B) \right)} \times \text{ enc}(C) \right)\); its three components have \(8\), \(20\) and \(20\) terms:

\begin{verbatim}
comp_1 = -2*a21*b11*c12 + 2*a11*b21*c12 - 2*a22*b21*c12 + 2*a21*b22*c12
         - 2*a12*b11*c21 + 2*a11*b12*c21 - 2*a22*b12*c21 + 2*a12*b22*c21

comp_2 = a12*b11*c11 + a21*b11*c11 - a11*b12*c11 + a22*b12*c11 - a11*b21*c11
         + a22*b21*c11 - a12*b22*c11 - a21*b22*c11 - 2*a21*b12*c12
         + 2*a12*b21*c12 + 2*a21*b12*c21 - 2*a12*b21*c21 - a12*b11*c22
         - a21*b11*c22 + a11*b12*c22 - a22*b12*c22 + a11*b21*c22
         - a22*b21*c22 + a12*b22*c22 + a21*b22*c22

comp_3 = a12*b11*c11 - a21*b11*c11 - a11*b12*c11 + a22*b12*c11 + a11*b21*c11
         - a22*b21*c11 - a12*b22*c11 + a21*b22*c11 - 2*a21*b12*c12
         + 2*a12*b21*c12 - 2*a21*b12*c21 + 2*a12*b21*c21 - a12*b11*c22
         + a21*b11*c22 + a11*b12*c22 - a22*b12*c22 - a11*b21*c22
         + a22*b21*c22 + a12*b22*c22 - a21*b22*c22
\end{verbatim}

The factors of \(2\) here are expected and are \textbf{not} a symptom of the retired encoding: on the traceless part \(a_{22} = - a_{11}\), so \(d_{a} = a_{11} - a_{22} = 2a_{11}\). The article's claim that no \(2\) appears is about the \(\overline{v}\)-coordinates of \thmref{Example}{ex-wk-signed}, which is a different expansion in different variables. Both routes agree exactly (anc/code/regenerate\_displays\_S.sage, anc/code/verify\_lorentzian\_encoding\_identity.sage).

\subsection{\texorpdfstring{The three assertions of \thmref{Remark}{rem-why-D}}{The three assertions of {[}rem-why-D{]}}}\label{app-remark2}

\textbf{The ideal is unchanged by the twist.} The three components of \({\overline{w}}_{2}\), which are also the entries of \(S\left( {\overline{w}}_{2} \right)\) up to sign, are \begin{equation}\fitdisplay{- v_{1,3}v_{2,2} + v_{1,2}v_{2,3},\quad v_{1,3}v_{2,1} - v_{1,1}v_{2,3},\quad - v_{1,2}v_{2,1} + v_{1,1}v_{2,2},}\end{equation} so \(\text{ideal}\left( {\overline{w}}_{2} \right) = \text{ ideal}\left( S{\overline{w}}_{2} \right)\) generator by generator: two of the three are negated and the third is fixed, and negating a generator does not change the ideal it generates. Confirmed independently by Gröbner-basis equality in anc/code/verify\_encoding\_gl3\_equivariance.sage.

\textbf{The Lorentzian form.} With \(A_{0}\) the traceless part of a generic \(A\), the polynomial \(- {\text{enc}\left( A_{0} \right)}_{1}^{2} - {\text{ enc}\left( A_{0} \right)}_{2}^{2} + {\text{ enc}\left( A_{0} \right)}_{3}^{2} - 4\det(A_{0})\) is identically \(0\) --- not zero at sampled points, but the zero polynomial in \(a_{11},\ldots,a_{22}\).

\textbf{The involution, and the \(\mathbb{F}_{3}\) congruence.} \(S^{2} = I\) exactly, while \(\tau^{2} = \text{ diag}(1,1,4) \neq I\) and \(\det\tau = - 2\). Over \(\mathbb{F}_{3}\), with \(T = \begin{pmatrix}
0 & 0 & 1 \\
1 & 1 & 0 \\
1 & 2 & 0
\end{pmatrix}\) of determinant \(1\), \begin{equation}\fitdisplay{T^{\top}T = \text{ diag}(2,2,1) = \text{ diag}( - 1, - 1,1)\quad\text{ in }\mathbb{F}_{3},}\end{equation} so \(S\)'s form is congruent to the identity form there. This is why the article's claim is stated for formally real fields rather than for fields in which \(- 1\) is a non-square: the latter condition is the wrong one, and \(\mathbb{F}_{3}\) is the counterexample.

\subsection{\texorpdfstring{The generators of \(I_{6}\)}{The generators of I\_\{6\}}}\label{app-i6}

\thmref{Theorem}{thm-mingens} gives \(I_{k}\) three minimal generators for every \(k\). At \(k = 6\) they are degree-\(6\) forms in the \(18\) variables \(v_{1,1},\ldots,v_{6,3}\), with \(32\) terms each and \textbf{every coefficient equal to \(\pm 1\)}:

\begin{verbatim}
w_6,1 = -v1_3*v2_1*v3_1*v4_1*v5_1*v6_2 + v1_1*v2_3*v3_1*v4_1*v5_1*v6_2
        - v1_3*v2_2*v3_2*v4_1*v5_1*v6_2 + v1_2*v2_3*v3_2*v4_1*v5_1*v6_2
        - v1_2*v2_1*v3_2*v4_3*v5_1*v6_2 + v1_1*v2_2*v3_2*v4_3*v5_1*v6_2
        + v1_3*v2_1*v3_3*v4_3*v5_1*v6_2 - v1_1*v2_3*v3_3*v4_3*v5_1*v6_2
        - v1_3*v2_1*v3_1*v4_2*v5_2*v6_2 + v1_1*v2_3*v3_1*v4_2*v5_2*v6_2
        - v1_3*v2_2*v3_2*v4_2*v5_2*v6_2 + v1_2*v2_3*v3_2*v4_2*v5_2*v6_2
        + v1_2*v2_1*v3_1*v4_3*v5_2*v6_2 - v1_1*v2_2*v3_1*v4_3*v5_2*v6_2
        + v1_3*v2_2*v3_3*v4_3*v5_2*v6_2 - v1_2*v2_3*v3_3*v4_3*v5_2*v6_2
        + v1_2*v2_1*v3_1*v4_1*v5_1*v6_3 - v1_1*v2_2*v3_1*v4_1*v5_1*v6_3
        + v1_3*v2_2*v3_3*v4_1*v5_1*v6_3 - v1_2*v2_3*v3_3*v4_1*v5_1*v6_3
        + v1_2*v2_1*v3_2*v4_2*v5_1*v6_3 - v1_1*v2_2*v3_2*v4_2*v5_1*v6_3
        - v1_3*v2_1*v3_3*v4_2*v5_1*v6_3 + v1_1*v2_3*v3_3*v4_2*v5_1*v6_3
        + v1_3*v2_1*v3_1*v4_2*v5_3*v6_3 - v1_1*v2_3*v3_1*v4_2*v5_3*v6_3
        + v1_3*v2_2*v3_2*v4_2*v5_3*v6_3 - v1_2*v2_3*v3_2*v4_2*v5_3*v6_3
        - v1_2*v2_1*v3_1*v4_3*v5_3*v6_3 + v1_1*v2_2*v3_1*v4_3*v5_3*v6_3
        - v1_3*v2_2*v3_3*v4_3*v5_3*v6_3 + v1_2*v2_3*v3_3*v4_3*v5_3*v6_3
\end{verbatim}

The other two components, \(w_{6,2}\) and \(w_{6,3}\), have the same shape and the same \(32\)-term, \(\pm 1\) profile; all three are printed in full in anc/code/regenerate\_displays\_S.sage's transcript, which ships as \texttt{anc/runs/display-ground-truth-S.txt}. The script asserts the coefficient profile rather than merely printing it, so a regression to the retired encoding fails the run instead of producing a plausible-looking table.

\subsection{The non-degeneracy witness}\label{app-witness}

The following table evaluates the point constructed in the proof of \thmref{Lemma}{lem-nondeg}.

\begin{verbatim}
j    witness (v_1,...,v_j)              w_j
2    (e1,e2)                            e3
3    (e1,e2,e1)                         e2
4    (e1,e2,e1,e1)                      e3
5    (e1,e2,e1,e1,e1)                   e2
6    (e1,e2,e1,e1,e1,e1)                e3
7    (e1,e2,e1,e1,e1,e1,e1)             e2
8    (e1,e2,e1,e1,e1,e1,e1,e1)          e3
\end{verbatim}

Each level is checked twice in anc/code/verify\_nondegeneracy\_witness.sage: once as the induction runs, and once by feeding the finished witness back through \thmref{Definition}{def-wk}'s recursion as a point of \(\mathbb{A}^{3j}\), which is how \thmref{Lemma}{lem-nondeg} states it. The same script runs the retired \(\tau\) alongside, where the witness must alternate, \(\left( e_{1},e_{2},e_{1},e_{2},\ldots \right)\), and the images are \((0,2,0)\), \((0,0,2)\), \((0,4,0)\), \(\ldots\), doubling at every level and reaching \(8\) by \(j = 8\). That column reproduces values recorded before the encoding changed, which is what makes the \(S\) column above a check rather than merely a new computation.

\subsection{Independence of the generators}\label{app-independence}

This is the corroboration cited in the proof of \thmref{Lemma}{lem-mingens-indep}. For each \(k\), the three components of \({\overline{w}}_{k}\) are expanded over the monomials that actually occur and the resulting \(3 \times N\) coefficient matrix is checked to have rank \(3\) --- which is exactly the assertion that no nontrivial constant relation holds among them:

\begin{longtable}[]{@{}cccccc@{}}
\toprule\noalign{}
\endhead
\bottomrule\noalign{}
\endlastfoot
\(k\) & \#variables & degree & \#monomials & rank & independent? \\
2 & 6 & 2 & 6 & 3 & yes \\
3 & 9 & 3 & 12 & 3 & yes \\
4 & 12 & 4 & 24 & 3 & yes \\
5 & 15 & 5 & 48 & 3 & yes \\
6 & 18 & 6 & 96 & 3 & yes \\
\end{longtable}

The monomial count doubles at each step while the rank stays at \(3\), so the margin widens rather than narrows as \(k\) grows (anc/code/verify\_generator\_independence.sage).


\section{References}

\begin{thebibliography}{99}
\bibitem{BrunsVetter1988} W. Bruns, U. Vetter, \emph{Determinantal Rings}, Lecture Notes in Mathematics 1327, Springer-Verlag, 1988. Out of print; freely available with Springer's permission at \href{https://www.home.uni-osnabrueck.de/wbruns/brunsw/detrings.pdf}{the first author's own page}.  \bibitem{EagonNorthcott1962} J.A. Eagon, D.G. Northcott, \emph{Ideals defined by matrices and a certain complex associated with them}, Proc. Roy. Soc. London Ser. A 269(1337) (1962), 188-204. DOI: \href{https://doi.org/10.1098/rspa.1962.0170}{10.1098/rspa.1962.0170}.  \bibitem{MotzkinTaussky1955} T. Motzkin, O. Taussky, \emph{Pairs of matrices with property L. II}, Trans. Amer. Math. Soc. 80 (1955), 387-401.  \bibitem{Stacks00OS} \emph{The Stacks Project}, tag 00OS (dimension formula for finitely generated field extensions / equicodimensionality of affine domains), and equivalently D. Eisenbud, \emph{Commutative Algebra with a View Toward Algebraic Geometry}, Springer GTM 150, 1995, Cor. 13.4. Open access at \href{https://stacks.math.columbia.edu/tag/00OS}{stacks.math.columbia.edu/tag/00OS}.  \bibitem{Gerstenhaber1961} M. Gerstenhaber, \emph{On dominance and varieties of commuting matrices}, Ann. of Math. 73 (1961), 324-348.  \bibitem{ArtinHochsterConjecture} M. Artin, M. Hochster, attributed conjecture: for \(I\) the ideal generated by the entries of \(XY - YX\), \(\text{Spec}(R/I)\) is reduced and Cohen-Macaulay. No single dateable original publication found; stated as ``Conjecture 1.1 (M. Artin, M. Hochster)'' in Majidi-Zolbanin and Snapp below.  \bibitem{MajidiZolbaninSnapp} M. Majidi-Zolbanin, B. Snapp, \emph{A note on the variety of pairs of matrices whose product is symmetric}, in \emph{Commutative Algebra and Its Connections to Geometry}, Contemporary Mathematics 555, American Mathematical Society, 2011, pp. 145-150. Author's open-access copy at \href{https://people.math.osu.edu/snapp.14/SymmetrizingVariety.pdf}{people.math.osu.edu/ snapp.14/SymmetrizingVariety.pdf}.  \bibitem{Chau2024} T. Chau, \emph{The \(F\)-regularity of algebraic sets related to commutator matrices}, J. Algebra 652 (2024), 52-66. DOI: \href{https://doi.org/10.1016/j.jalgebra.2024.04.008}{10.1016/j.jalgebra.2024.04.008}.  \bibitem{Hreinsdottir2005} F. Hreinsdóttir, \emph{Conjectures on the ring of commuting matrices}, arXiv:math/0501465. Open access at \href{https://arxiv.org/abs/math/0501465}{arxiv.org/abs/math/0501465}.  \bibitem{Knutson2003} A. Knutson, \emph{Some Schemes Related to the Commuting Variety}, J. Algebraic Geom. 14 (2005), 283-294; preprint arXiv:math/0306275. (The arXiv abstract page's ``to appear'' note long predates publication.) Open access at \href{https://arxiv.org/abs/math/0306275}{arxiv.org/abs/math/0306275}.  \bibitem{Youcis2014} A. Youcis, answer to ``What can we say about the image of a regular map?'', Mathematics Stack Exchange, question 731178, answer 732415, 2014. States both the point-set-topology fact that the continuous image of an irreducible space is irreducible, and the \((x,y) \mapsto (x,xy)\) non-closed-image example cited above. CC BY-SA 3.0. Open access at \href{https://math.stackexchange.com/questions/731178/what-can-we-say-about-the-image-of-a-regular-map}{math.stackexchange.com/questions/731178}.  \bibitem{Stacks031R} \emph{The Stacks Project}, tag \textbf{031R} = Lemma 10.157.3, Serre's criterion for reducedness: ``Let \(R\) be a Noetherian ring. The following are equivalent: (1) \(R\) is reduced, and (2) \(R\) has properties \(\left( R_{0} \right)\) and \(\left( S_{1} \right)\).'' Open access at \href{https://stacks.math.columbia.edu/tag/031R}{stacks.math.columbia.edu/tag/031R}. Note that the enclosing \textbf{section} 10.157 is tag 031O and is titled ``Serre's criterion for \emph{normality}''; its own Lemma 10.157.4 (tag 031S) is the normality criterion \(\left( R_{1} \right)\) and \(\left( S_{2} \right)\), a different statement. 

\bibitem{Tambour2000} T. Tambour, \emph{The Number of Solutions of Some Equations in Finite Groups and a New Proof of Itô's Theorem}, Communications in Algebra \textbf{28} (2000), no. 11, 5353--5361. \href{https://doi.org/10.1080/00927870008827160}{doi:10.1080/00927870008827160}. The published version is not open access; freely available as the earlier \emph{On the number of solutions of some equations in finite groups}, Stockholm University Research Reports in Mathematics No. 15 (1998), at \href{https://www2.math.su.se/reports/1998/15/}{www2.math.su.se/reports/1998/15}. 

\end{thebibliography}
\end{document}